\documentclass[12pt]{extarticle}
\usepackage{amsmath, amsthm, amssymb, hyperref, color, verbatim}
\usepackage{graphicx}
\usepackage{caption}
\usepackage{subcaption}
\usepackage{bm}
\usepackage[T1]{fontenc}
\usepackage[font=small,labelfont=bf,tableposition=top]{caption}
\usepackage{epstopdf}

\DeclareCaptionLabelFormat{andtable}{#1~#2  \&  \tablename~\thetable}

\usepackage{floatrow}
\newfloatcommand{capbtabbox}{table}[][\FBwidth]

\usepackage{blindtext}

\usepackage[shortlabels]{enumitem}
\usepackage{graphicx}
\usepackage[all]{xypic}
\usepackage{makecell}
\usepackage{algpseudocode}
\usepackage{lipsum}

\oddsidemargin \evensidemargin
\newcommand{\leoncomment}[1]{ {\textcolor{blue}{ #1  --Leon}}} 

\newtheorem{theorem}{Theorem}
\numberwithin{theorem}{section}
\newtheorem{proposition}[theorem]{Proposition}
\newtheorem{lemma}[theorem]{Lemma}

\newtheorem{definition}[theorem]{Definition}
\newtheorem{problem}[theorem]{Problem}
\newtheorem{remark}[theorem]{Remark}
\newtheorem{example}[theorem]{Example}

\newtheorem{algorithm}[theorem]{Algorithm}

\newcommand{\RR}{\mathbb{R}}
\newcommand{\R}{\mathbb{R}}

\newcommand{\TP}{\mathbb{TP}}

\newcommand{\Trop}{\text{Trop}}
\newcommand{\tvol}{\text{tvol\,}}
\newcommand{\tdet}{\text{tdet\,}}
\DeclareMathOperator*{\argmin}{arg\,min}

 \date{}
 
\title{\textbf{Tropical Principal Component Analysis\\
 and its Application to Phylogenetics}}

\author{Ruriko Yoshida \and Leon Zhang \and Xu Zhang}

\begin{document}

\maketitle

\begin{abstract}
Principal component analysis is a widely-used method for the dimensionality reduction of a given data set in a high-dimensional Euclidean space. Here we define and analyze two analogues of principal component analysis in the setting of tropical geometry. In one approach, we study the Stiefel tropical linear space of fixed dimension closest to the data points in the tropical projective torus; in the other approach, we consider the tropical polytope with a fixed number of vertices closest to the data points. We then give approximative algorithms for both approaches and apply them to phylogenetics, testing the methods on simulated phylogenetic data and on an empirical dataset of Apicomplexa genomes.

\end{abstract}

\section{Introduction}
Principal component analysis (PCA) is a popular and robust method for reducing the
dimension of a high-dimensional data set. 
Given a positive integer $s\in \mathbb N$ and a collection of data points in a high-dimensional Euclidean space $\mathbb R^e$, the procedure projects the data points onto a plane of fixed dimension $s-1$, which is obtained by minimizing the sum of squared distances between
each point in the dataset and its orthogonal projection onto the
plane.  
This linear plane is a vector translate of some $(s-1)$-dimensional linear subspace; PCA also constructs an orthonormal basis for that subspace whose vectors are called principal
  components. The low-dimensional plane is thus described by an $(s\times e)$-dimensional matrix, whose first $s-1$ rows are the principal components and whose last row is the translation vector.

In this paper we propose two analogous approaches to a principal component analysis in the setting of tropical geometry. Given a positive integer $s$ and a collection of data points in the tropical projective torus, our tropical principal component analyses seek a tropically-geometric object, as close as possible to the data points in the tropical metric $d_{tr}$. In both cases, furthermore, this tropically-geometric object will be described by an $(s\times e)$-dimensional matrix. 

Classically, a full-rank matrix of shape $(s\times e)$ with $s<e$ defines an $s$-dimensional linear subspace of $\mathbb R^e$ via the span of its rows. This subspace is also described by the Pl\"ucker coordinates of the matrix. Tropically, on the other hand, these two notions diverge: the tropical Pl\"ucker coordinates of a tropical matrix produce a \emph{Stiefel tropical linear space}, defined in \cite{FR}, while the row-span of the matrix yields a \emph{tropical polytope}. These two notions give rise to our two interpretations of tropical principal component analysis.
\begin{figure}[!ht]
\centering
\includegraphics[width=2in]{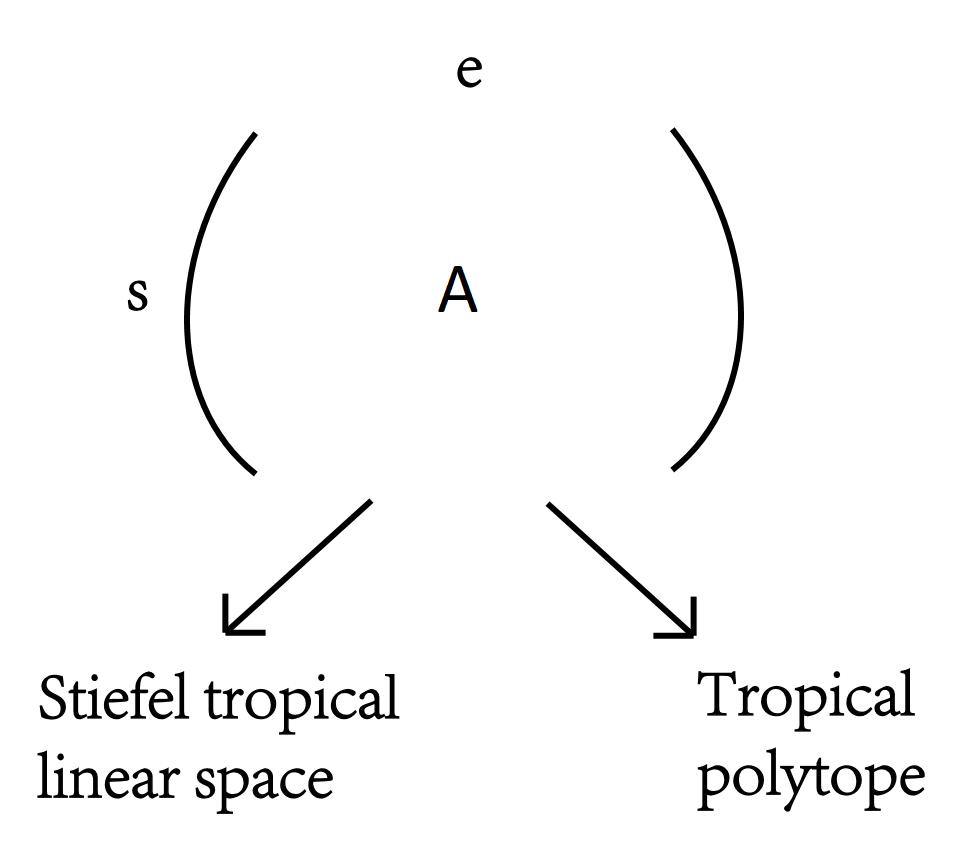}
\caption{A tropical matrix $A$ gives rise to both a Stiefel tropical linear space and a tropical polytope.}
\end{figure}

In Section \ref{def:trop}, we
discuss the basics of tropical geometry that we use throughout this paper. We also review the interpretation of the space of equidistant trees as a tropical linear space. 

We then describe our first approach to a tropical principal component analysis in Section \ref{trop:lin}, as the {Stiefel tropical linear space} closest to the data points under the tropical metric $d_{tr}$. We give an exact description for an $(e-1)$st order tropical PCA of $e$ points in terms of the \emph{tropical volume}, originally introduced in \cite{DGJ}. We also describe an heuristic algorithm to approximate a best-fit Stiefel tropical linear space of a given dimension.

Next, in Section \ref{trop:poly}, we discuss a tropical principal component analysis in terms of best-fit {tropical polytopes}. We reformulate the problem of finding a best-fit tropical polytope in terms of a mixed integer programming problem, then describe an approximative algorithm similar to the above.

We then apply these methods to phylogenetics. The space of rooted equidistant phylogenetic trees with $m$ leaves is naturally embedded into a tropical projective torus as a tropical linear space, so that collections of phylogenetic trees form a natural tropical dataset. We apply the approximative algorithms for both methods of tropical PCA on a simulated phylogenetic dataset in Section \ref{sim}, and on an empirical dataset of Apicomplexa genomes in Section \ref{apicomplexa}. In our tropical polytope approach, equidistant trees remain ultrametrics after projection, and so we examine the distribution of tree topologies in that case.
\section{Tropical basics}\label{def:trop}
In this section we review some basics of tropical geometry. Interested readers should consult \cite{MS} or \cite{J} for more details.

In the tropical semiring $(\,\RR \cup \{-\infty\},\oplus,\odot)\,$, the basic
arithmetic operations of addition and multiplication are redefined as follows:
$$a \oplus b := \max\{a, b\}, ~~~~ a \odot b := a + b ~~~~\text{  where } a, b \in \RR.$$
The element $-\infty$ is the identity element for addition and 0 is the identity element for multiplication: for all $a\in \RR\cup \{-\infty\}$, we have $a \oplus -\infty = a$ and $a \odot 0 = a .$

Given two $(m\times k)$-dimensional matrices $A,B$ and an $(n\times m)$-dimensional matrix $C$ with entries in $\RR\cup\{-\infty\}$, we can define the tropical matrix operations $A\oplus B$ and $A\otimes C$ in analogy with the ordinary matrix operations. Namely,
\[(A\oplus B)_{ij} = A_{ij}\oplus B_{ij},\ \  (A\otimes C)_{ij} = \bigoplus_{\ell=1}^m A_{i\ell}\otimes B_{\ell j}.\]

If $m = k$, so that $A$ is a square matrix, we can also define its \emph{tropical determinant} in analogy with the classical operation. We have
\[\tdet A = \bigoplus_{\sigma\in S_m} \left(\bigotimes_{i=1}^n A_{i,\sigma(i)}\right).\]

If the tropical determinant of $A$ is attained by at least two distinct permutations in $S_e$, we say that $A$ is \emph{tropically singular}.

In tropical geometry we often work in the \emph{tropical projective torus} $\mathbb R^e \!/\mathbb R {\bf 1}$, where ${\bf 1}$ denotes the all-ones vector. Given two points $v, w$ in
the tropical projective torus, their {\em tropical
  distance} $d_{\rm tr}(v,w)$ is defined as follows:
\begin{equation}
\label{eq:tropmetric} d_{\rm tr}(v,w) \,\, = \,\,
\max \bigl\{\, |v_i - w_i  - v_j + w_j| \,\,:\,\, 1 \leq i < j \leq e \,\bigr\} ,
\end{equation}
where $v = (v_1, \ldots , v_e)$ and $w= (w_1, \ldots , w_e)$.
This metric is also known as the
{\em generalized Hilbert projective metric} 
\cite[\S 2.2]{AGNS}, \cite[\S 3.3]{CGQ}.

\begin{example}
We present three points $P_1, P_2,$ and $P_3$ in the tropical projective torus in Figure \ref{trop_dist_Leon}.
\begin{figure}[!ht]
\centering
\includegraphics[width=2.5in]{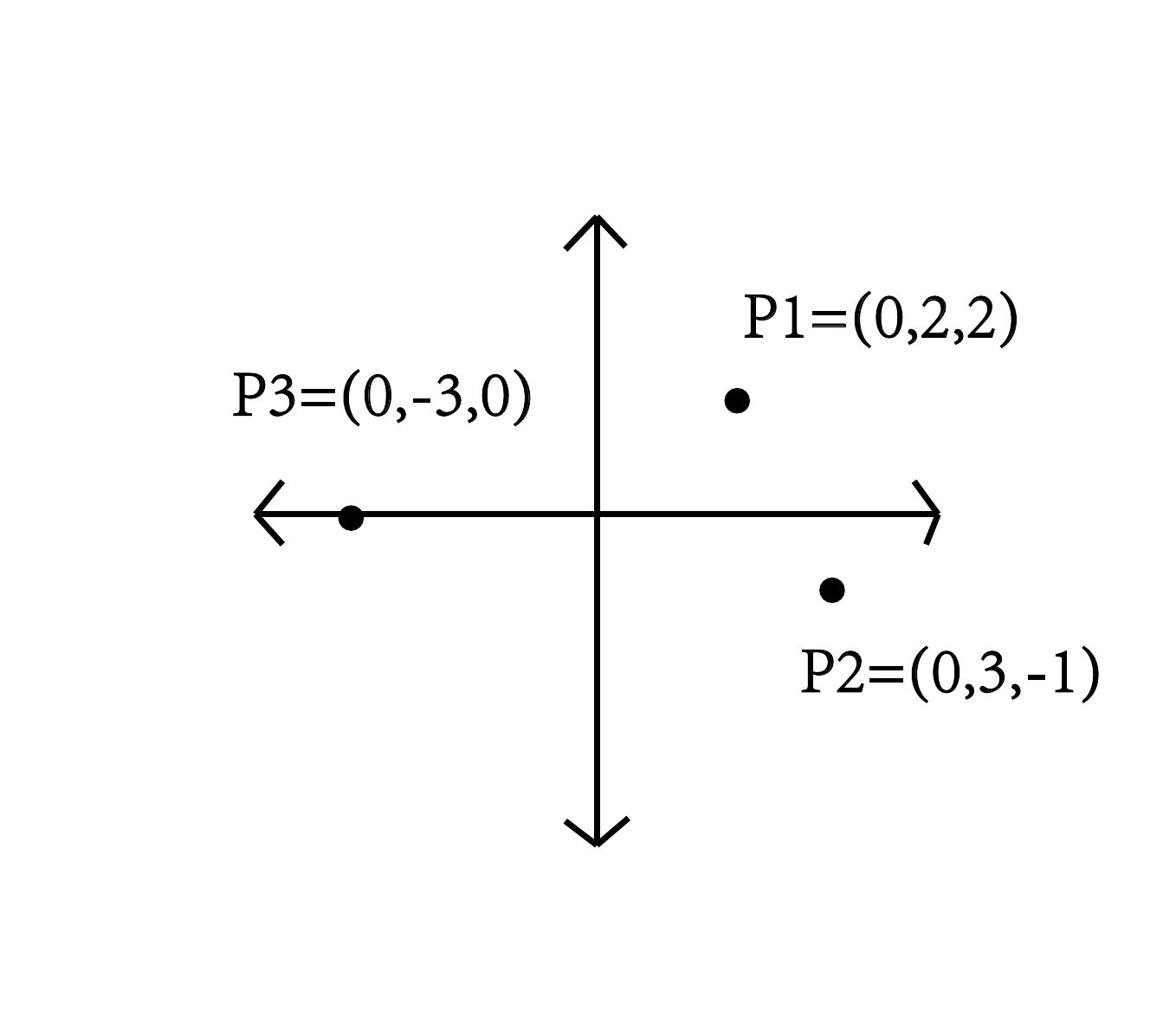}
\caption{Tropical distance in tropical plane}
\label{trop_dist_Leon}
\end{figure}
It can be checked that $d_{tr}(P_1, P_2) = 4, d_{tr}(P_1, P_3) = 5$, and $d_{tr}(P_2, P_3) = 7$.
\end{example}

\begin{example} 
There is a natural embedding of a phylogenetic tree on $m$ leaves as a point in $\RR^{\binom{m}{2}}/\RR {\bf 1}$, discussed in Section \ref{tropical-interpretation-for-phylogenetics}, in which the coordinates of a tree give the distances between leaves. We can think of the tropical distance between two phylogenetic trees as measuring the ``range'' of the disagreement between the two trees' distances.

For example, suppose we have the two phylogenetic trees $v = (4, 4, 2)$ and $w = (2, 4, 2)$ as in Figure \ref{trop_dist_Xu}. The largest disagreement between $v$ and $w$ in which tree $v$ finds a longer distance between two leaves is $\max \{v_i - w_i\} = 2$, and the largest disagreement between $v$ and $w$ in which tree $w$ shows a bigger distance between two leaves is $\max\{w_j - v_j\} = 2$. So $d_{\rm tr}(v,w) = 2 + 2 = 4$.
\begin{figure}[!ht]
\centering
\includegraphics[width=2.75in]{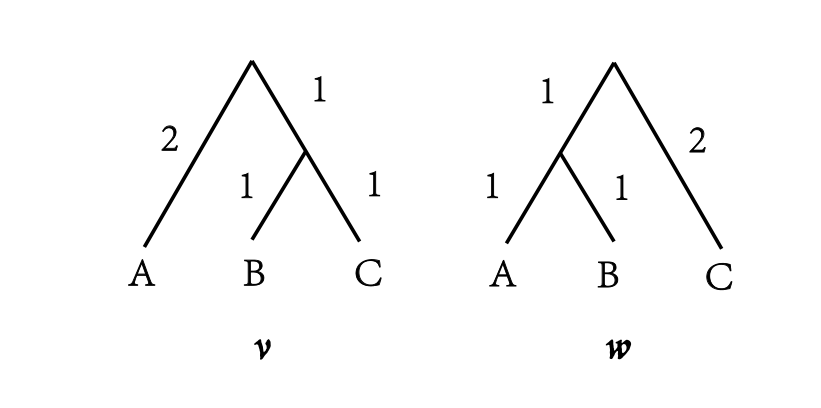}
\caption{Tropical distance in phylogenetics}
\label{trop_dist_Xu}
\end{figure}
\end{example}
\subsection{Tropical convexity}
\label{trop:conv}

We next review some basic definitions from tropical convexity.

A subset $S$ of $\RR^e$ is called \textit{tropically convex} if it contains the point $a \odot  x \oplus b \odot y$ for all $x,y
\in S$ and all $a, b \in \RR$. The \textit{tropical convex hull} or \textit{tropical polytope} of a given subset $V \subset \RR^e$ is the smallest tropically convex subset containing $V$ of $\RR^e$. We write it as ${\rm tconv}(V)$. The tropical convex hull of $V$ coincides with the set of all tropical linear combinations
$${\rm tconv}(V) = \{a_1 \odot v_1 \oplus a_2 \odot v_2 \oplus \cdots \oplus a_r \odot v_r : v_1,\ldots,v_r \in V \text{ and } a_1,\ldots,a_r \in \RR\}.$$

Any tropically convex subset $S$ of $\RR^e$ is closed under tropical
scalar multiplication, $\RR \odot S \subseteq S$. In other words, if $x
\in S$ then $x+\lambda \cdot {\bf 1} \in S \text{ for all } \lambda \in
\RR$. We therefore identify the tropically convex set $S$ with
its quotient in the tropical projective torus
$\RR^e/\RR {\bf 1}$.

Let $\mathcal P$ be a tropical polytope
$\mathcal{P} = {\rm tconv}(D^{(1)},D^{(2)},\ldots,D^{(s)})$, where the $D^{(i)}$ are points in $\RR^e/\RR {\bf 1}$. There is a projection map $\pi_{\mathcal P}$ sending any point $D$ to a closest point in the tropical polytope $\mathcal{P}$ as
\begin{equation}
\label{eq:tropproj} 
 \pi_\mathcal{P} (D) \,= \,
\lambda_1 \odot  D^{(1)} \,\oplus \,
\lambda_2 \odot  D^{(2)} \,\oplus \, \cdots \,\oplus \,
\lambda_s \odot  D^{(s)}  ,
\quad {\rm where} \,\, \lambda_k = {\rm min}(D-D^{(k)}) , \,
 k = 1,\ldots,s.
\end{equation}
This formula appears as \cite[Formula 5.2.3]{MS}.

\subsubsection{Tropical linear spaces}

Our treatment of this topic largely follows \cite[Sections 3 and 4]{JSY}.

\begin{definition}
Let $p:[e]^d\to \RR\cup\{-\infty\}$ be a map satisfying the following conditions:
\begin{enumerate}
	\item $p(\omega)$ depends only on the unordered set $\omega=\{\omega_1,\ldots,\omega_d\}\subseteq [e]$,
	\item $p(\omega)=-\infty$ whenever $\omega$ has fewer than $d$ elements, and
	\item({Exchange relation}.) For any $(d-1)$-subset $\sigma$ and any $(d+1)$-subset $\tau$ of $[e]$, the maximum
	\[\max\{p(\sigma\cup\{\tau_i\})+p(\tau-\{\tau_i\}):i\in[d+1]\}\]
	is attained at least twice.
\end{enumerate}
Such a map $\pi$ is called a \emph{tropical Pl\"ucker vector}.
\end{definition}

\begin{definition}
Let $p:[e]^d\to  \RR\cup\{-\infty\}$ be a tropical Pl\"ucker vector. The \emph{tropical linear space} $L_p$ consists of all points $x\in \mathbb T\mathbb P^{e-1}$ such that, for any $(d+1)$-subset $\tau$ of $[e]$, the maximum of the numbers $p(\tau-\{\tau_i\})+x_{\tau_i}$, for $i=1,\ldots, d$, is attained at least twice.
\end{definition}

It is well-known \cite[Proposition 5.2.8]{MS} that tropical linear spaces are tropically convex. 

\begin{definition}
\label{stiefel-tropical-linear-space}
Let $A$ be a tropical $d\times e$ matrix. Given a $d$-sized subset $\omega\subseteq [e]$, we write $A_\omega$ for the $d\times d$ matrix whose columns are the columns of $A$ indexed by elements of $\omega$. Then the map
\[p:[e]^d\to\mathbb R\cup \{-\infty\}\]
\[ \omega\mapsto \tdet(A_\omega)\]
is a tropical Pl\"ucker vector. The corresponding tropical linear space is called the \emph{Stiefel tropical linear space} given by $A$.
\end{definition}

\begin{example}
\label{tropical-linear-space-example}
Let
\[A = \left(\begin{matrix}0 & 2 & 4\\
0 & -1 & -3\end{matrix}\right)\]
and let $p$ be its associated tropical Pl\"ucker vector. Then 
\[p(\{1,2\})=\tdet\left(\begin{matrix}0&2\\0&-1\end{matrix}\right)=2,\]
\[p(\{1,3\})=\tdet\left(\begin{matrix}0&4\\0&-3\end{matrix}\right)=4,\]
\[p(\{2,3\})=\tdet\left(\begin{matrix}2&4\\-1&-3\end{matrix}\right)=3.\]
The Stiefel tropical linear space corresponding to $A$ is a tropical line in $\RR^3/\RR {\bf 1}$. It is pictured in Figure \ref{pic_example3_4}.
\begin{figure}[!ht]
\centering
\includegraphics[width=2.5in]{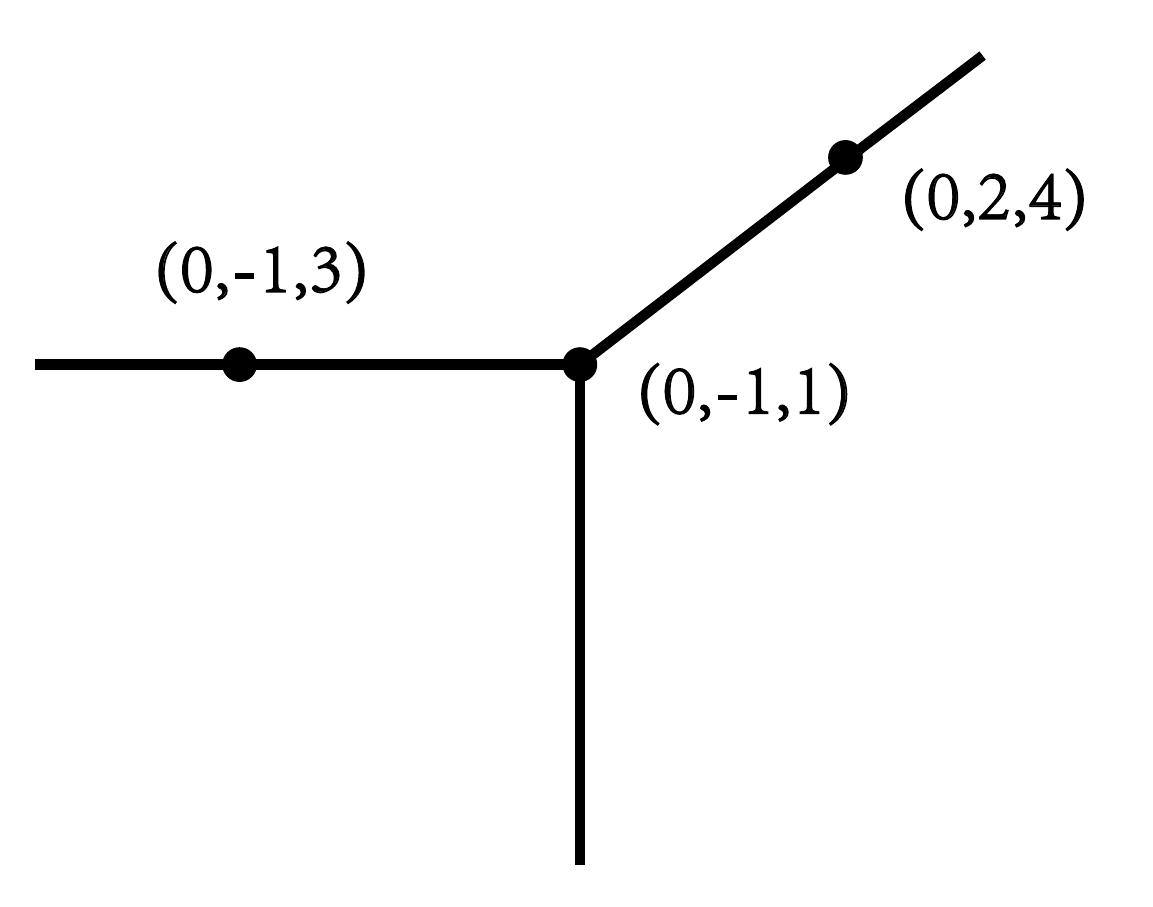}
\caption{Stiefel tropical linear space in Example \ref{tropical-linear-space-example}}
\label{pic_example3_4}
\end{figure}
\end{example}

In order to consider a ``tropical principal component analysis'', as described at the beginning of this section, we must be able to project onto a tropical linear space. This projection operation is described by the \emph{Red and Blue Rules}. From \cite[Theorem 15]{JSY} we have:

\begin{theorem}[The Blue Rule]
\label{blue-rule}
Let $p:[e]^d\to \overline \RR$ be a tropical Pl\"ucker vector and $L_p$ its associated tropical linear space in $\RR^e/\RR {\bf 1}$. Fix $u\in\RR^e/\RR {\bf 1}$, and define the point $w\in\RR^e/\RR {\bf 1}$ whose $i$th coordinate is
\begin{equation}
\label{eq:bluerule}
 \quad w_i \,\, = \,\,
 {\rm max}_\tau \,{\rm min}_{j \not\in \tau} \bigl( u_j + p({\tau \cup \{i\}}) - p({\tau \cup \{j\}}) \bigr)
 \qquad {\rm for} \,\,\, i = 1,2,\ldots, e
 \end{equation}
where $\tau$ runs over all $(d-1)$-subsets of $[e]$ that do not contain $i$.

Then $w\in L_p$, and any other $x\in L_p$ satisfies $d(u,x)\geq d(u,w)$. In other words, $w$ attains the minimum distance of any point in $L_p$ to $u$.
\end{theorem}

\begin{theorem}[The Red Rule]
\label{red-rule}
Let $p:[e]^d\to \overline \RR$ be a tropical Pl\"ucker vector and $L_p$ its associated tropical linear space in $\RR^e/\RR {\bf 1}$. Fix $u\in\RR^e/\RR {\bf 1}$. Let $v$ be the all-zeros vector. For every $(d+1)$-sized subset $\tau$ of $[e]$, compute $\max p({\tau-\tau_i}) + u_{\tau_i}$. If this maximum is unique, attained with index $\tau_i$, then let $\gamma_{\tau,\tau_i}$ be the positive difference between the second maximum and this maximum, and set $v_{\tau_i}=\max(v_{\tau_i}, \gamma_{\tau,\tau_i})$.

Then $v$ gives the difference between $u$ and a closest point of $L_p$. In particular, if $w$ is the point in $L_p$ returned by the Blue Rule, we have 
\[u = w + v.\]
\end{theorem}

We write $\pi_{L_p}$ as the projection function which takes a point $u\in \RR^e/\RR {\bf 1}$ and returns the nearest point $w\in L_p$ given by the Blue Rule.

\begin{example}
\label{red-blue-rule-example}
Let $A$ be the matrix of Example \ref{tropical-linear-space-example}, with $p$ and $L_p$ its associated tropical Pl\"ucker vector and Stiefel tropical linear space. Let $u$ be the point $(1, -2, 3)\in\RR^e/\RR {\bf 1}$. 

The Blue Rule constructs a point $w\in\RR^3/\RR{\bf 1}$ whose first coordinate is
\[\max(\min(u_1+p(\{1,2\})-p(\{1,2\}), u_3+p(\{1,2\})-p(\{1,3\})),\]
\[\min(u_1+p(\{1,3\})-p(\{1,3\}),u_2+p(\{1,3\})-p(\{2,3\}))).\]
Substituting in, we get the first coordinate of $w$ as 
\[w_1 = \max(\min(1 + 2 -2, 3 + 2 - 4), \min(1 + 4 - 4, -2 + 4 - 3)) = \max(1, -1) = 1.\]
Similarly, we get $w_2 = -2$ and $w_3 = 0$. So the Blue Rule outputs the vector $(1, -2, 2)$.

The Red Rule constructs a vector $v$ as follows. First, we begin with $v = (0, 0, 0)$. Next we take the set $\tau = [e]$ and compute $\max(p(\{2,3\})+u_1, p(\{1,3\})+u_2, p(\{1,2\})+u_3) = \max(3 + 1, 4 -2, 2 + 3) = 5.$ So the Red Rule redefines $v_3 = 5 - 4 = 1$, and hence outputs the vector $v=(0,0,1)$. Now Theorem \ref{red-rule} states that $u = w + v$, which is easily verified to hold.
\end{example}

\begin{definition}
Let $v = (v_1,\ldots, v_e)$ be a real vector, and define the tropical linear functional $\bigoplus (-v_i)\otimes x_i$. Let $\mathcal H$ be the tropical solution set of this linear functional: that is, $\mathcal H$ consists of all $x\in\RR^e/\RR {\bf 1}$ such that the maximum of $\bigoplus( -v_i)\otimes x_i$ is attained at least twice. We call any $\mathcal H$ obtained in this way a \emph{tropical hyperplane}.
\end{definition}

\begin{remark}
\label{tropical-hyperplane-valuated-matroid}
Let $A$ be a tropical matrix of dimensions $(e-1)\times e$. Then the Stiefel tropical linear space of $A$ is a tropical hyperplane. Furthermore, any tropical hyperplane is the Stiefel tropical linear space of such a tropical matrix $A$.
\end{remark}

\subsection{A tropical interpretation for phylogenetic trees}
\label{tropical-interpretation-for-phylogenetics}

In this section we describe some of the tropical aspects underlying the study of phylogenetic trees. Our treatment of this subject largely follows \cite[Section 4.3]{MS}.

\begin{definition}
A \emph{dissimilarity map} $d$ is a function $d:[m]\times [m]\to \mathbb R_{\geq 0}$ such that $d(i,i)=0$ and $d(i,j)=d(j,i)\geq 0$ for each $i,j\in [m]$. If, furthermore, we have that $d(i,j)\leq d(i,k)+d(k,j)$ for all $i,j,k\in [m]$, we call $d$ a \emph{metric}. Note that for convenience we often write $d_{ij}$ for the term $d(i,j)$.

We can represent a dissimilarity map $d$ by an $m\times m$ matrix $D$ whose $(i,j)$th entry is $d_{ij}$. Because $D$ is clearly symmetric and all diagonal entries are trivial, there is a natural embedding of $d$ into the tropical space $\mathbb R^{e} = R^{\binom m 2}$.
\end{definition}

In fact, the condition of being a metric is intrinsically tropical.

\begin{lemma}
\label{metric-condition}
Let $d:[m]\times [m]\to\mathbb R_{\geq 0}$ be a dissimilarity metric and $D$ its corresponding matrix. Then $d$ is a metric iff $-D\odot -D = -D$.
\end{lemma}
\begin{proof}
The $(i,j)$th entry of $-D\odot -D$ is equal to 
\[\bigoplus_{k=1}^m -d_{ik}-d_{kj} = \max_{k\in [m]}(-d_{ik}-d_{kj}).\]
In particular, we note that the $(i,j)$th entry of $-D\odot -D$ is at least as large as $-d_{ij}=-d_{ii}-d_{ij}$. Now a simple negation of the definition shows that $d$ is a metric iff $-d_{ij}\geq \max_{k\in [m]}(-d_{ik}-d_{kj})$. 
\end{proof}

\begin{definition}
Let $T=(V,E)$ be a tree with $m$ labeled leaves and no vertices of degree two. We call such a tree a \emph{phylogenetic tree}.
\end{definition}

\begin{definition}
Let $T$ be a phylogenetic tree with $m$ leaves labeled with the elements of $[m]$, and assign a length $\ell_e\in \mathbb R$ to each edge $e$ of $T$. Let $d:[m]\times [m]\to \mathbb R$ be defined so that $d_{ij}$ is the total length of the unique path from leaf $i$ to leaf $j$. We call a function $d$ obtained in this way a \emph{tree distance}. If, furthermore, each entry of the distance matrix $D$ is nonnegative, then $d$ is in fact a metric. We call such a $d$ a \emph{tree metric}. As before, we can embed $D$ into $\mathbb R^{e}$. 
\end{definition}

Of course, any tree distance differs from a tree metric by some scalar multiple of ${\bf 1}$. Hence the sets of tree distances and tree metrics coincide in $\mathbb R/{\bf 1}\RR$. 

\begin{definition}
Let $d:[m]\times [m]\to \mathbb R_{\geq 0}$ be a metric which satisfies the following strengthening of the triangle inequality for each choice of $i,j,k\in [m]$:
\[d(i,k)\leq \max(d(i,j),d(j,k)).\]
We call such a metric an \emph{ultrametric}. Let $\mathcal U_m$ denote the collection of all ultrametrics in $\mathbb R^e/{\bf 1}\mathbb R$.
\end{definition}

It is well-known that all ultrametrics are tree metrics. In fact, all ultrametrics are derived from \emph{equidistant trees}, where all leaves have the same distance to some distinguished root vertex. Furthermore, the tree metric of an equidistant tree is an ultrametric; hence ultrametrics and equidistant trees convey equivalent information.

Let $L_m$ denote the subspace of $\mathbb R^e$ defined by the linear equations $x_{ij} - x_{ik} + x_{jk}=0$ for $1\leq i < j <k \leq m$. The tropicalization $\Trop(L_m)\subseteq \RR^e/\RR {\bf 1}$ is the tropical linear space consisting of points $(v_{12},v_{13},\ldots, v_{m-1,m})$ such that $\max(v_{ij},v_{ik},v_{jk})$ is obtained at least twice for all triples $i,j,k\in [m]$.

\begin{remark}
Experts in tropical geometry will note that the tropical linear space $\Trop(L_m)$ corresponds to the graphic matroid of the complete graph $K_m$.
\end{remark}

\begin{theorem}
\label{ultrametrics}
The image of $\mathcal U_m$ in the tropical projective torus $\RR^e/\RR {\bf 1}$ coincides with $\Trop(L_m)$.
\end{theorem}
\begin{proof}
Let $(v_{12},v_{13},\ldots, v_{m-1,m})\in \Trop(L_m)$. We may assume that each coordinate is nonnegative, so that this point corresponds to the image of a dissimilarity map $d$. To see that $d$ is in fact an ultrametric, fix $i,j,k\in [m]$. We know that $\max(d_{ij},d_{ik}, d_{kj})$ is attained at least twice, by the definition of $\Trop(L_m)$. If $d_{ij}$ is one of these maximums then it must be equal to $\max(d_{ik},d_{kj})$. If $d(i,j)$ is not one of these maximums then it must be strictly less than $\max(d_{ik},d_{kj})$. Either way, we have that $d_{ij}\leq \max(d_{ik},d_{kj})$. This shows that $d$ is in fact an ultrametric, so that $\Trop(L_m)\subseteq \mathcal U_m$.

Let $\bar D\in \mathcal U_m$. Then there exists some lifted ultrametric $d$ with associated matrix $D$. Fix a choice of $i,j,k\in [m]$, and without loss of generality let $i,j$ such that $d_{ij}=\max(d_{ij},d_{ik},d_{kj})$. Because $d$ is an ultrametric, we have that $d_{ij}\leq \max(d_{ik},d_{kj})$.  Hence in fact $d_{ij}=\max(d_{ik},d_{kj})$, and the maximum of $\max(d_{ij},d_{ik},d_{kj})$ is attained at least twice. Thus $\Trop(L_m)\supseteq \mathcal U_m$.
\end{proof}

In words, Theorem \ref{ultrametrics} states that the image of the space of labeled rooted trees is a tropical linear space. The set of equidistant trees thus has an intrinsic tropical structure.

\section{Tropical PCA as a Stiefel tropical linear space}\label{trop:lin}

As noted in the introduction, one can interpret ordinary $(s-1)$th principal component analysis as a method of dimensionality reduction, replacing data points with their projections onto the translate of some particularly well-fitting linear space of dimension $s-1$. Classically, this translation of a well-fitted linear space can be described by an $(s\times e)$-dimensional matrix, whose first $(s-1)$ rows are the basis vectors of the linear space, and whose last row is a translation vector from the origin.

In analogy with the classical case, our approach to an $(s-1)$th tropical principal component analysis is to replace data points with their tropical projections onto the best-fit Stiefel tropical linear space of dimension $(s-1)$, defined by a tropical matrix of size $s\times e$.

\subsection{Best-fit tropical hyperplanes}

We begin our discussion of tropical principal component analysis by
considering a specific case: reducing by one the dimension
of a collection of $e$ datapoints in $\RR^e/\RR {\bf 1}$. In other words, we
seek the $(e-1)$th order tropical PCA, or a \emph{best-fit tropical hyperplane}, for a collection of $e$ data points in $\RR^e/\RR {\bf 1}$.

We require the following definition, from \cite{DGJ}.

\begin{definition}
Let $A$ be an $e\times e$ matrix with entries in $\overline \RR$ whose rows correspond to $e$ points in $\RR^e/\RR {\bf 1}$. The \emph{tropical volume} of $A$ is given by the expression
\[\tvol A := \bigoplus_{\sigma\in S_e} \sum a_{i,\sigma(i)} - \bigoplus_{\tau\in S_e-\sigma_{opt}}\sum a_{i,\tau(i)},\]
where $\sigma_{opt}$ is an optimal permutation attaining the tropical determinant in the first tropical sum.
\end{definition}

Recall that a square tropical matrix $A$ is \emph{tropically singular} if two distinct permutations attain the tropical determinant. The following, from \cite[Lemma 5.1]{RGST}, is one of the earliest results in tropical geometry:

\begin{lemma}
\label{common-hyperplane}
Let $A$ be an $e\times e$ tropical matrix whose rows represent $e$ points of $\RR^e/\RR {\bf 1}$. Then $A$ is tropically singular iff those $k$ points lie on a tropical hyperplane in $\RR^e/\RR {\bf 1}$. In particular, $\tvol(A)=0$ iff the $e$ points lie on a common tropical hyperplane.
\end{lemma}

Of course, if our collection of $e$ datapoints $D^{(i)}$ lie on a common hyperplane, then this hyperplane is our $(e-1)$th tropical PCA. This fact hints at some relationship between the tropical volume and the best fit hyperplane. In fact, this relationship is quite strong.

\begin{theorem}
\label{tropical-volume}
Let $D^{(1)},\ldots, D^{(e)}$ be a collection of $e$ points in $\RR^e/\RR {\bf 1}$. Then the best-fit hyperplane attains a distance from the $e$ points equal to their tropical volume, and one such best-fit hyperplane is spanned by a choice of $e-1$ of the points.
\end{theorem}

To prove this theorem, we first show that the tropical volume is an upper bound on the minimal distance of a best-fit tropical hyperplane.

\begin{lemma}\label{upper-bound}
Let $D^{(1)},\ldots, D^{(e)}$ be a collection of $e$ points in $\RR^e/\RR {\bf 1}$, and let $A$ be the matrix whose $i,j$th entry is $D^{(i)}_j$. Then there exists a hyperplane of distance $\tvol A$ from the data points, spanned by some choice of $e-1$ of the points.
\end{lemma}
\begin{proof}
Suppose that all $e$ data points can be spanned by a single hyperplane. Then Lemma \ref{common-hyperplane} tells us that this best-fit hyperplane is of distance $\tvol A = 0$ from the data points.

Now suppose that the $e$ data points do not lie on the same hyperplane. 
Without loss of generality, we may assume that the data points $D^{(1)},\ldots, D^{(e)}$ are ordered so that $\sigma_{opt}$ in the above definition of the tropical volume is just the identity, and hence the tropical determinant is attained along the diagonal of $A$.

Let $\rho$ attain the second maximum in the above definition of the tropical volume. Since $\rho$ is not the identity, there must exist some $j$ such that $\rho(j)\neq j$. 
Let $A'$ be the matrix obtained by deleting the $j$th row from $A$, and let $p$ and $\mathcal H$ the tropical Pl\"ucker vector and tropical hyperplane corresponding to $A'$ as in Example \ref{tropical-linear-space-example}. The total distance from $H$ to our data points is just the distance from $H$ to $D^{(j)}$, as all other data points are on $H$ by construction.

We compute the difference vector between $D^{(j)}$ and its projection onto $H$ using the Red Rule (Theorem \ref{red-rule}). The only possible choice for an $e$-sized subset $\tau$ of $[e]$ is just $\tau = [e]$, and we need to compute the maximum and second-maximum values of $p({[e]-\tau_i}) + D^{(j)}_{\tau_i}$, taken over all choices of $\tau_i\in [e]$. For any such $\tau_i$, we note that $p({[e]-\tau_i}) + D^{(j)}_{\tau_i}$ is equal to
\[\bigoplus_{\sigma\in S_d,\; \sigma(j)=\tau_i} \sum_i D_{\sigma(i)}^{(i)}.\]
That is, $p_{\tau-\tau_i}+D_{\tau_i}^{(j)}$ is the tropical sum of all permutations which map $\tau_i$ to $j$. In particular, $\tau_i = j$ must yield the largest choice of $p_{\tau-\tau_i}+D_{\tau_i}^{(j)}$, and the second-largest choice must be attained by $\tau_i=\rho^{-1}(j)$. Hence the Red Rule implies that the distance between $D^{(j)}$ and its projection is just the tropical volume, as desired.
\end{proof}

\begin{remark}
In general, a best-fit Stiefel tropical linear space need not be unique. For example, in the proof of Lemma \ref{upper-bound}, there clearly must be at least two indices $j$ such that $\rho(j)\neq j$.
\end{remark}

We next show that the tropical volume is also an upper bound. To do so, we first derive some intermediate results.

\begin{lemma}
\label{same-tropical-volume}
Let $D^{(1)},\ldots, D^{(e)}$ be a collection of $e$ points in $\RR^e/\RR {\bf 1}$, and let $A$ be the $e\times e$ tropical matrix whose $i,j$th entry is $D^{(i)}_j$. Define the matrix $A'$ whose $i,j$th entry equals $p([e]-\{i\}) + D^{(j)}_i$. Then $A$ and $A'$ have the same tropical volume.
\end{lemma}
\begin{proof}
We note that $A'$ is obtained from $A$ by transposition then adding some multiple of {\bf 1} to each row. Both of these operations preserve the tropical volume of a matrix.
\end{proof}

Now suppose that $\mathcal H$ is a tropical hyperplane in $\RR^e/\RR {\bf 1}$, and write its corresponding tropical Pl\"ucker vector as $p([e]-\{i\})$. We can calculate the distance $\delta_j(\mathcal H)$ of $\mathcal H$ from the $j$th datapoint $D^{(j)}$ by the Red Rule: the distance is given by
\[\delta_j(\mathcal H) = \max_i(p([e]-\{i\}) + D^{(j)}_i) - \text{2ndmax}_i (p([e]-\{i\}) + D^{(j)}_i).\]
We write the total distance of $\mathcal H$ from our datapoints as $d(\mathcal H)$. It is given by
\[d(\mathcal H)=\sum_j \delta_j(\mathcal H) = \sum_j \left(\max_i(p([e]-\{i\}) + D^{(j)}_i) - \text{2ndmax}_i (p([e]-\{i\}) + D^{(j)}_i)\right).\]
We can rewrite the cost function $d(\mathcal H)$ above by grouping together the summed and subtracted terms. For fixed $j$, define $\alpha_j(\mathcal H) = \max_i(p([e]-\{i\}) + D^{(j)}_i)$ and $\beta_j(\mathcal H) = \text{2ndmax}_i (p([e]-\{i\}) + D^{(j)}_i)$. Then $\delta_j(\mathcal H) = \alpha_j(\mathcal H)-\beta_j(\mathcal H)$, and the cost function can also be written as
\[d(\mathcal H) = \sum_j \delta_j(\mathcal H) =\sum_j \alpha_j(\mathcal H) - \sum_j \beta_j(\mathcal H).\]

\begin{definition}
Fix $j$ in the cost function above, and let $i_1$ and $i_2$ be distinct indices such that $\alpha_j(\mathcal H) = p([e]-\{i_1\})+D^{(j)}_{i_1}$ and $\beta_j(\mathcal H) = p([e]-\{i_2\})+D^{(j)}_{i_2}$. If $\delta_j(\mathcal H)=0$, meaning that $\alpha_j(\mathcal H)=\beta_j(\mathcal H)$, we say that the two indices $i_1$ and $i_2$ \emph{appear in a tie} for index $j$. If there exists another index $i_3$ such that $p([e]-\{i_3\}) + D^{(j)}_{i_3}=\alpha_j(\mathcal H) = \beta_j(\mathcal H)$, we call this a \emph{multiple tie} for index $j$; if there does not exist such an $i_3$, we call this a \emph{two-way tie}.
\end{definition}

Note that, in the event of a tie, we may choose any two of the indices attaining the tie to correspond to $\alpha_j(\mathcal H)$ and $\beta_j(\mathcal H)$.

\begin{lemma}
Let $\mathcal H$ be an optimal hyperplane in $\RR^e/\RR {\bf 1}$, and let $p$
be its corresponding tropical Pl\"ucker vector. Choose an index $i$ such that $p({[e]-\{i\}})<\beta_j\leq \alpha_j$ for all $j$. Then we can perturb $\mathcal H$ to obtain a new best-fit hyperplane $\mathcal H'$ so that $p([e]-\{i\})+D_i^{(j)}=\beta_j$ for some $j$, and this $j$ corresponds to a multiple tie.
\label{hyperplane-perturbation}
\end{lemma}
\begin{proof}
Because $p({[e] - \{i\}})$ does not appear in the cost function by
assumption, by Remark \ref{tropical-hyperplane-valuated-matroid} we
can find a new hyperplane $\mathcal H'$ with the same tropical Pl\"ucker vector as $\mathcal H$ except for a larger value for $p({[e] - \{i\}})$. 

If we make $p({[e] - \{i\}})$ large enough, it must appear in the cost function for $\mathcal H'$. In fact, it must appear as part of a multiple tie. If it were a second maximum not equal to the maximum, then $\mathcal H'$ would be a better-fitting hyperplane.
\end{proof}

\begin{lemma}\label{lower-bound}
Let $A$ be an $e\times e$ matrix with entries in $\mathbb R\cup\{-\infty\}$ whose rows correspond to points in $\RR^e/\RR {\bf 1}$, and let $A'$ be constructed from $A$ as in Lemma \ref{same-tropical-volume}. Then the tropical volume of $A$ is a lower bound for the cost function. Furthermore, we have that $\sum {\alpha_j}=\tdet A'$.
\end{lemma}
\begin{proof}
Let $\mathcal H$ be a best-fit hyperplane in $\RR^e/\RR {\bf 1}$ for the rows of $A$, with corresponding tropical Pl\"ucker vector $p$. The basic argument is as follows: we can perturb $\mathcal H$ to obtain a new best-fit hyperplane whose sum of distances to the data points given by the Red Rule is the difference of two permutations, with the larger permutation corresponding to the tropical determinant of $A'$.

We prove the result by induction on $e$. For the base case, let $e = 1$. Then the tropical volume and the cost function are both trivial.

Suppose we have proved the lemma up to $e-1$. We divide the situation into several possible cases. First, let there be some index $k$ appearing only in ties in the cost function, with at most one of these appearances being a two-way tie. If $k$ appears in a two-way tie, let $D^{(j_k)}$ denote the corresponding datapoint. Otherwise, let $D^{(j_k)}$ denote some datapoint for which $k$ appears in a multiple tie. 

Then we can write the cost function as
\[p({[e]-\{k\})} + D^{(j_k)}_k - p({[e]-\{k\}}) - D^{(j_k)}_k + \sum_{j\neq j_k} \delta_j(\mathcal H).\]

Construct the matrix $A''$ by deleting the $k$th row and $j_k$th column from $A'$. We also define the hyperplane $\mathcal H'\subseteq \RR^{e-1}/\RR {\bf 1}$ obtained by ``deleting'' the index $\{k\}$ from $[e]$: the tropical Pl\"ucker vector $p'$ corresponding to $\mathcal H'$ is defined by
\[p'([e-1]-\{i\}) = \begin{cases} p([e]-\{i\}) & \mbox{if }i<k\\ p([e] - \{i+1\}) & \mbox{if } i\geq k\end{cases}.\] 

Because we assumed that $k$ appears in at most one two-way tie, for any $j\neq j_k$ we can choose the indices corresponding to $\alpha_j(\mathcal H)$ and $\beta_j(\mathcal H)$ so that $k$ does not appear in $\alpha_j(\mathcal H)-\beta_j(\mathcal H)=\delta_j(\mathcal H)$. By construction, therefore, $d(\mathcal H) = \sum_{j\neq j_k} \delta_j(\mathcal H)$ is also the distance between $\mathcal H'\subseteq \RR^{e-1}/\RR {\bf 1}$ and the rows of the matrix $A''$. Furthermore, the optimality of $\mathcal H$ implies that $\mathcal H'$ must be a best-fit tropical hyperplane for the rows of $A''$.

In particular, the inductive hypothesis states that $d(\mathcal H') =\sum_{j\neq j_k} \delta_j(\mathcal H)$ is bounded from below by the tropical volume of $A''$. It also implies that $\sum_{j\neq j'}\alpha_j(\mathcal H) = \tdet A''$. It therefore follows that $d(\mathcal H)=d(\mathcal H')$ is bounded below by a difference of distinct permutations in $A'$, and that $\sum \alpha_j(\mathcal H)$ equals a sum of terms of $A'$ corresponding to some permutation of $S_e$. 

In fact, since each $\alpha_j$ is the largest term in the $j$th row of $A'$, we must have that $\sum \alpha_j = \tdet A'$. Hence we have for some $\sigma\in S_e$,
\[d(\mathcal H)\geq \tdet A' - \sum_i a_{i,\sigma(i)}\geq \tvol A' = \tvol A\]
where the last equality holds by Lemma \ref{same-tropical-volume}.

Now suppose that there exists an index $k$ such that $p({[e]-\{k\}})$ does not appear in any terms in the cost function. Then by Lemma \ref{hyperplane-perturbation}, we may replace $\mathcal H$ with another hyperplane such that $k$ appears only in a multiple tie for some index $j$. We are now in the previous case, and the same argument holds as before.

Finally, suppose that for each index $i$, either $p({[e] - \{i\}})$ appears in the cost function as part of a non-tie, or $p({[e] - \{i\}})$ appears in at least two two-way ties. Pick $i_1$ such that $\alpha_j(\mathcal H) = p([e]-\{i_1\})+D^{(j)}_{i_1}$ for $j$ corresponding to a non-tie. We write this index $j$ as $j_{i_1}$, and we write $i_0$ as the index corresponding to $\beta_{i_{j_1}}(\mathcal H)$. Suppose that there does not exist some other index $j_{i_2}$ such that $\beta_{j_{i_2}}=p({[e]-\{i_1\}})+D^{(j_{i_2})}_{i_2}$. Then we could perturb $\mathcal H$ by slightly lowering $p_{[e]-\{i_1\}}$ to obtain a better-fitting hyperplane, a contradiction. Hence such a $j_{i_2}$ must exist.

In fact, note that we can pick $j_{i_2}$ to avoid a multiple-way tie
at that index. Otherwise, perturbing $p({[e]-\{i_1\}})$ upward would
not affect the second and first minimum, and we could obtain the same contradiction. It follows that the index $j_{i_2}$ must correspond to either a two-way tie or a non-tie. In either case, therefore, there is a unique other index $i_2$ such that $\alpha_{j_{i_2}}=p([e]-\{i_2\}) +D^{(j_{i_2})}_{i_2}$.

If the cost function term corresponding to $j_{i_2}$ is a non-tie, and $i_2$ appeared in no other cost function terms as part of the subtracted term, then we can obtain a contradiction in a similar way as above by perturbing $p({[e]-\{i_2\}})$. If the cost function term corresponding to $j_{k_2}$ is a tie, and $i_1$ and $i_2$ appeared in no other cost function terms as part of the subtracted term, then we could obtain a contradiction in a similar way as above by perturbing $p({[e]-\{i_1\}})$ and $p({[e]-\{i_2\}})$ in sync.

Hence in a similar fashion we may obtain indices $i_3$, and a $i_4$, and so on, such that each $i_k = \alpha_{j_{i_k}}(\mathcal H)$ for some index $j_{i_k}$ corresponding to either a two-way tie or a non-tie. Because there can only be at most $e$ such indices $j_{i_k}$, there must exist $\ell$ and $\ell'$ such that $i_\ell = i_{\ell'}$ with $\ell>\ell'$. If $\ell'\neq 0$, then we may repeat the argument by perturbing $p([e]-\{i_{\ell'-1}\})$ upward, possibly in tandem with some earlier Pl\"ucker coordinates. Hence we must find $i_\ell =i_0$ for some $\ell$.

If $\ell < e$, and if there exists another index $i_{\ell+1}$ which appears as a positive term in the cost function, we repeat the above argument. It therefore follows that if $p([e]-\{i\})$ appears in the cost function as part of a non-tie, it must appear at least twice as part of a non-tie or a two-way tie. By assumption, therefore, each index appears at least twice as part of a non-tie or a two-way tie.


In particular, the pigeonhole principle implies that each index $i$ appears exactly twice as part of a non-tie or a two-way tie. It can thus be assumed that each index appears once as part of some $\alpha_i$ and once as part of some $\beta_i$. Now the distance function $d(\mathcal H)$ is the difference between two different permutations of $S_e$. As before, $\sum_{\alpha_i}$ must therefore equal the tropical determinant of $A'$, and the distance function $d(\mathcal H)$ must be bounded below by the tropical volume as desired.
\end{proof}

Together, Lemmas \ref{upper-bound} and \ref{lower-bound} imply Theorem \ref{tropical-volume}. This result provides a new interpretation for the tropical volume of a collection of $e$ points: it measures the deviation of those points from lying on a common hyperplane. It also suggests a possible extension of the definition of a tropical volume to rectangular matrices (\cite[Section 5]{DGJ}): the tropical volume of a ``skinny'' matrix with more rows than columns could be defined as the sum of the distances of the row-points from a best-fit tropical hyperplane.

If $n > e$, an optimist might hope that the best-fit tropical hyperplane of $n$ points in $\RR^e/\RR{\bf 1}$ would again attain the tropical volume of some subset of $e$ of those points. In fact, this does not hold even for $e = 3$:
\begin{example}\label{a34exam}
Consider the matrix $A$ whose rows correspond to data points in $\RR^3/\RR {\bf 1}$:
\[A = \begin{pmatrix}0 & -2 & -2\\ 0 & -1 & 2\\ 0 & 2 & -1\\ 0 & 2 & 2\end{pmatrix}.\]
The tropical volume of the first three points in $A$ is 4, so any tropical line must attain a distance at least 4 to the four points. This is attained by the tropical line with apex at $(0, 2, 2)$.
\end{example}

\begin{figure}[!ht]
\begin{floatrow}
\ffigbox{%
  \includegraphics[width=2in]{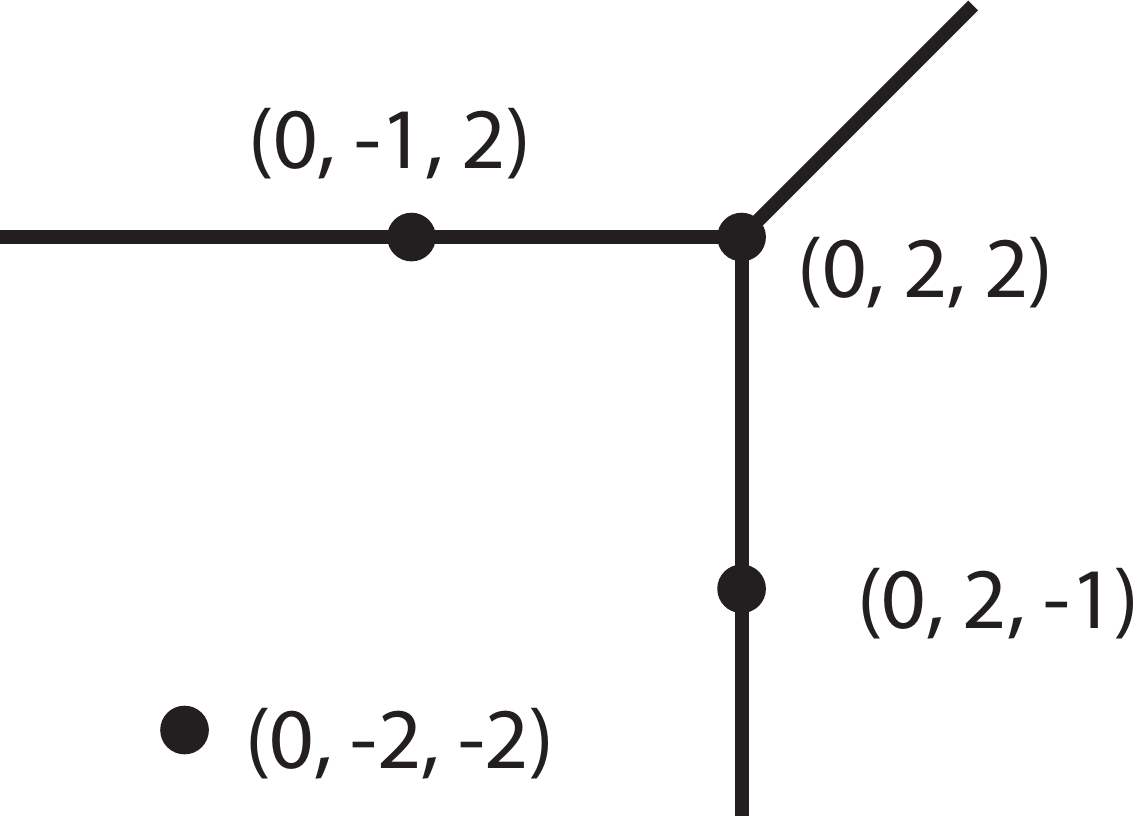}
}{%
  \caption{Example \ref{a34exam}}%
}
\ffigbox{%
  \includegraphics[width=2in]{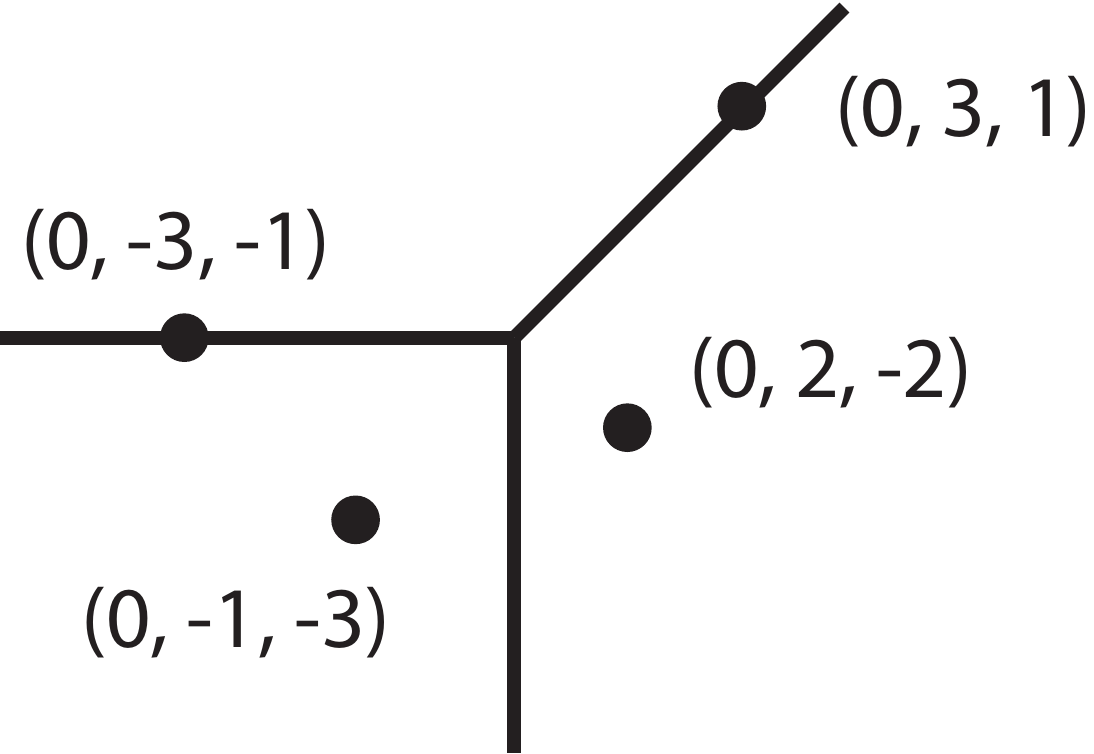}
}{%
  \caption{Example \ref{sad-example}}%
}
\end{floatrow}
\end{figure}
\begin{example}
\label{sad-example}
Let $A$ be the following matrix whose rows correspond to data points in $\RR^3/\RR {\bf 1}$:
\[A = \begin{pmatrix}0 & -1 & -3\\ 0 & 2 & -2\\0 & 3 & 1 \\ 0 & -3 & -1\end{pmatrix}.\]
The largest tropical volume of any choice of three rows is 2, but inspection shows that a best-fit tropical line attains a total distance of 3.
\end{example}

\subsection{Best-fit Stiefel tropical linear spaces}

In view of Theorem \ref{tropical-volume} and Lemma \ref{upper-bound}, we describe an algorithm to approximate a best-fit Stiefel tropical linear space of any given dimension. For simplicity, below we state the algorithm for a Stiefel tropical linear space of dimension 2.

\begin{algorithm}
\label{tropical-linear-space-algorithm}
\begin{algorithmic}
\State Stochastic optimization algorithm to fit $L_p$ to $D^{(i)}$.
\State Fix an ordered set $V = (D^{(1)}, D^{(2)}, D^{(3)})$ and compute $L_p(V)$. 
\Repeat:
	\State Sample three datapoints $D^{(j_1)}, D^{(j_2)}, D^{(j_3)}$ randomly from the set of all datapoints.
	\State Let $V' = \{D^{(j_1)}, D^{(j_2)}, D^{(j_3)}\}$.
	\State Compute $d(L_p(V'))$.
	\State if $d(L_p(V))>d(L_p(V')),$ set $V \gets V'$.
\Until convergence.
\end{algorithmic}
\end{algorithm}

This algorithm attempts to minimize $d(L_p)$ by randomly varying the three points generating $L_p$ within the set of all datapoints. Whenever a choice of three points $V'$ improves upon the current configuration $V$, we replace $V$ with $V'$. Convergence is assessed by considering whether a new choice of $V$ has been found over a fixed number of previous iterations; if no better $V$ is found over some prespecified number of iterations, then the algorithm terminates.

\begin{remark}
Algorithm \ref{tropical-linear-space-algorithm} does not always attain an exact best-fit tropical linear space. This is clear, for example, if we consider a variant of the algorithm for fitting a 0-dimensional Stiefel tropical linear space, i.e., a tropical Fermat-Weber point as in \cite{LY}. In general, the collection of tropical Fermat-Weber points for a given dataset need not contain a data point.
\end{remark}

Because the space of ultrametrics $\mathcal U_m$ is a tropical linear space (Theorem \ref{ultrametrics}), which are tropically convex, the convex hull of points in $\mathcal U_m$ is contained in $\mathcal U_m$. Unfortunately, however, the Stiefel tropical linear space defined by these points may not be contained in $\mathcal U_m$.

\begin{lemma}
\label{Stiefel-containment}
Let $L_p$ be a tropical linear space and $D^{(i)}\in L_p$ some points in the tropical linear space. Then it need not be the case that the Stiefel tropical linear space $L_q$ defined by the points is contained in $L_p$. 
\end{lemma}
\begin{proof}
For a very simple counterexample, let $L_p$ be the tropical line in $\RR^3/\RR {\bf 1}$ centered at the origin, and take the two points $D^{(1)} = (0, -1, 0)$ and $D^{(2)} = (0, -2, 0)$. We have the picture in Figure \ref{Stiefel_exam}.
\begin{figure}[!ht]
\centering
\includegraphics[width=2.5in]{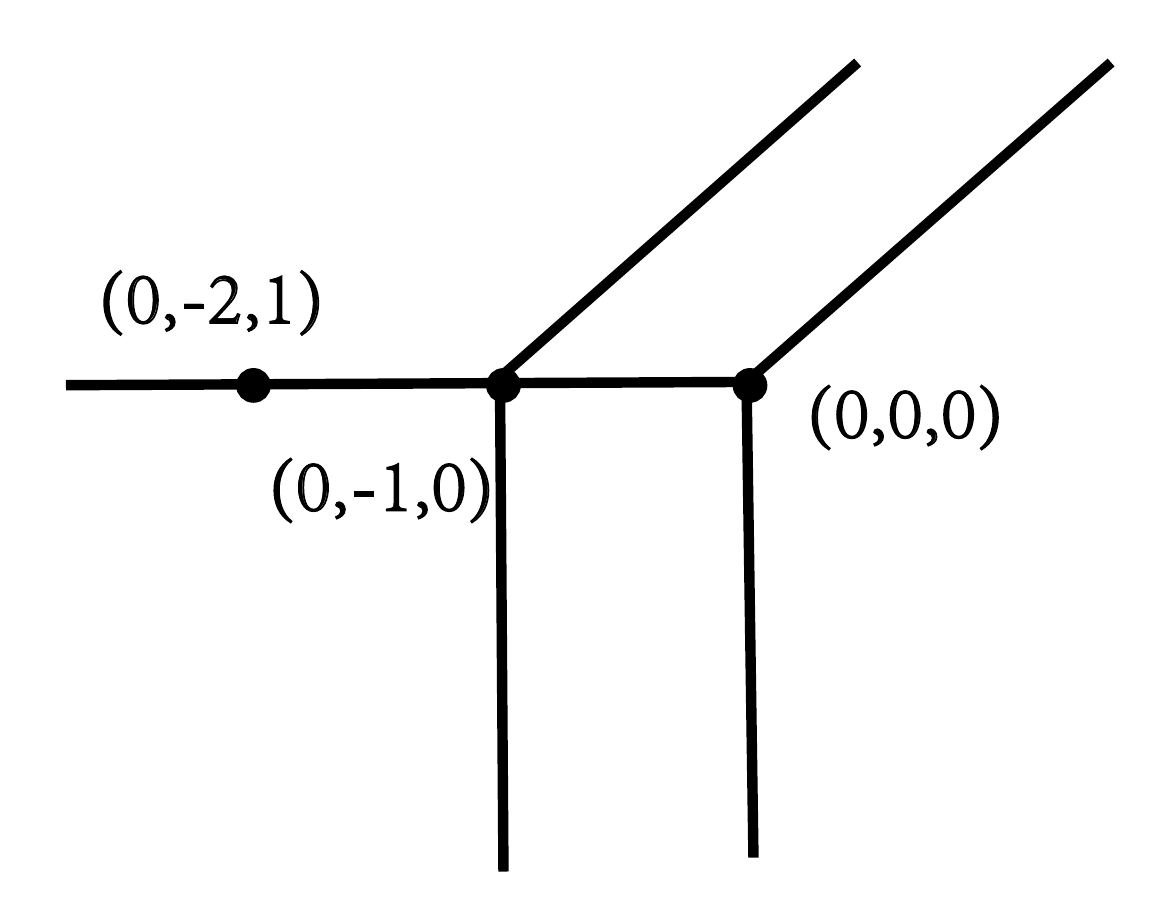}
\caption{Lemma \ref{Stiefel-containment}}
\label{Stiefel_exam}
\end{figure}
\end{proof}

If our data points $D^{(i)}$ correspond to ultrametrics, Lemma \ref{Stiefel-containment} tells us that the tropical linear space produced by Algorithm \ref{tropical-linear-space-algorithm} may not be contained in the overall space of ultrametrics. Hence this approach does not apply directly to the analysis of equidistant trees.

In the proof of Lemma \ref{Stiefel-containment}, however, if our two chosen points $D^{(1)}$ and $D^{(2)}$ lie on different rays of the tropical line $L_p$, it is easy to see that their corresponding Stiefel tropical linear space will be contained in $L_p$ as well. In general, given some points $D^{(i)}$ in a tropical linear space $L_p$, it would be interesting to study the conditions under which their corresponding Stiefel tropical linear space $L_q$ satisfies $L_q\subseteq L_p$. Such a result would enable a natural extension of these methods to the study of ultrametrics.

The classical principal components have a nested structure, in which the $i$th PCA is contained in the $(i+1)$st PCA for each $i$. It is natural to wonder whether a similar relationship holds in this tropical analogue. Again, the situation is complicated.

\begin{example}
\label{fermat-weber-noncontainment}
The minimal distance of a zeroth tropical principal component, or a \emph{tropical Fermat-Weber point}, is given in \cite[Theorem 3]{LY}.

Let $D^{(1)} = (0, -2, -2), D^{(2)} = (0, -1, 2),$ and $D^{(3)} = (0, 2, -1)$. Then their tropical volume equals 4, and a tropical Fermat-Weber point attains a total distance of seven from the three points. A best-fit hyperplane is given by the line with apex at $(0, 1, -2)$, and inspection shows that no point on this line is a Fermat-Weber point.

On the other hand, the point $(0, -1, -1)$ can be checked to be a Fermat-Weber point, and the line with apex at $(0, 2, -1)$ is a best-fit hyperplane containing that Fermat-Weber point. In other words, a best-fit tropical line need not fit a best-fit tropical point, but we can find an example in this case for which this containment holds.

\begin{figure}[!ht]
\centering
\includegraphics[width=2in]{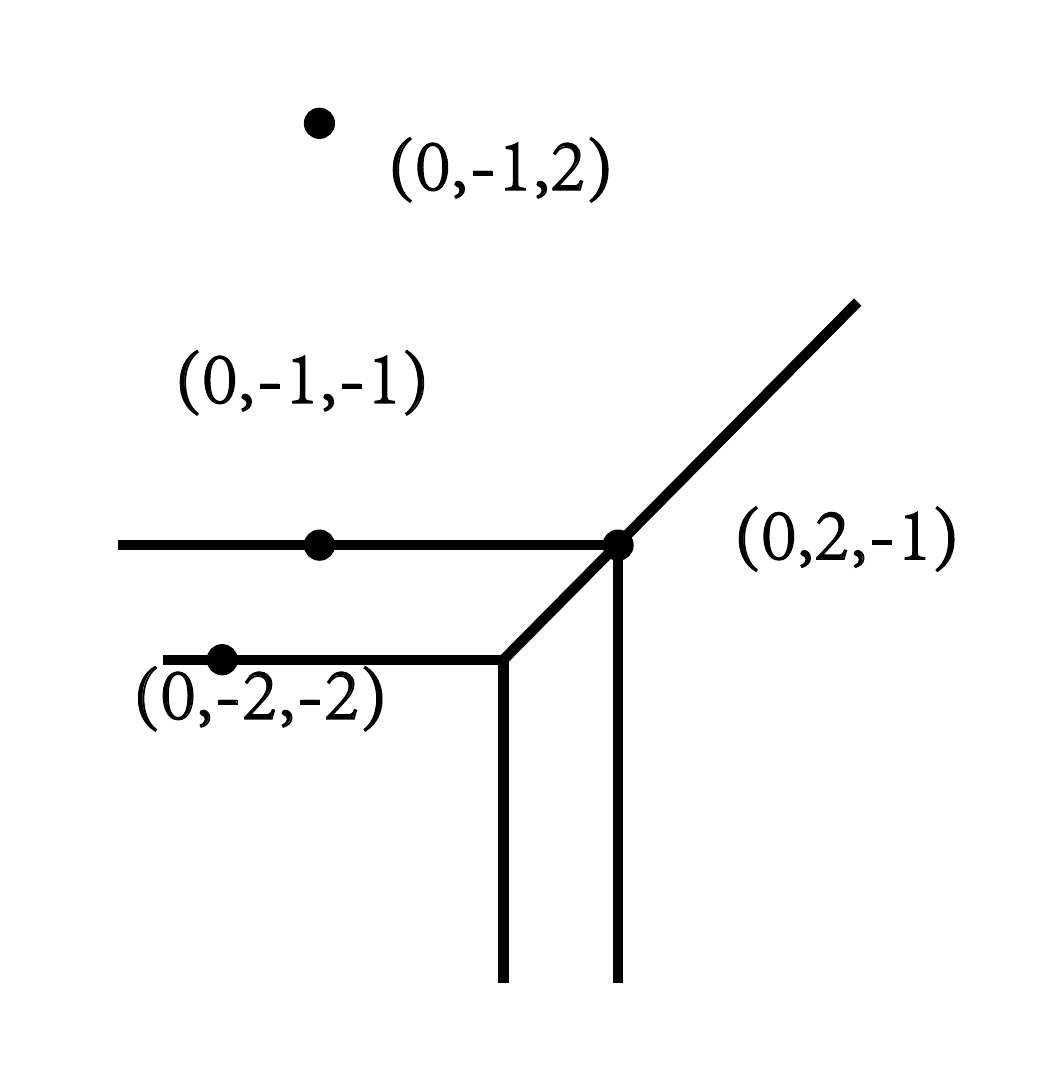}
\caption{Lemma \ref{fermat-weber-noncontainment}}
\label{FW-noncontainment-picture}
\end{figure}
\end{example}

\section{Tropical PCA as a tropical polytope}\label{trop:poly}

We now discuss a different notion of a tropical principal component analysis, in which our analogue to a linear plane is a tropical polytope. Classically, the row-span of a matrix of dimensions $s\times e$ defines a linear space of dimension at most $s$. In the tropical setting, by contrast, Section \ref{trop:conv} tells us that the row-span of a tropical matrix is a tropical polytope.

A tropical principal component analysis, therefore, outputs the tropical convex hull of $s$ points in $\RR^e/\RR {\bf 1}$ minimizing the distances between each
point in the sample to its projection onto the convex hull. For
simplification, we focus on the second order principal components, noting that the following discussion could be generalized to arbitrary $s$. Our problem can be written as follows:
\begin{problem}\label{optimization}
We seek a solution for the following optimization problem:
\[
\min_{D^{(1)}, D^{(2)}, D^{(3)} \in \RR^e/\RR{\bf 1}} \sum_{i = 1}^n  d_{\rm
  tr}(d_i, d'_i)
\]
where
\begin{equation}\label{const1:convex}
d'_i= \lambda_1^i \odot  D^{(1)} \,\oplus \,
\lambda_2^i \odot  D^{(2)} \,\oplus \,
\lambda_3^i \odot  D^{(3)}  ,
\quad {\rm where} \,\, \lambda_k^i = {\rm min}(d_i-D^{(k)}) ,
\end{equation}
and 
\begin{equation}\label{const2:convex}
d_{\rm
  tr}(d_i, d'_i) = \max\{|d_i(k) - d'_i(k) - d_i(l) + d'_i(l)|: 1 \leq
k < l \leq e\}
\end{equation}
with 
\begin{equation}\label{const3:convex}
d_i = (d_i(1), \ldots , d_i(e)) \text{ and } d'_i = (d'_i(1), \ldots , d'_i(e)).
\end{equation}
\end{problem}
 \bigskip 

In fact, this problem can be reformulated in terms of mixed integer programming.
\begin{proposition}\label{polytope_projection}
Problem
\ref{optimization} can be formulated as the following optimization problem:
\begin{align}
\text{minimize }& {}\sum_{i=1}^{n} \Delta_i &\\\nonumber
\text{subject to: } &\Delta_i \ge
d_{i}(k)-d^{'}_{i}(k)-d_{i}(l)+d^{'}_{i}(l), & 1 \le k < l \le e, i\in [n]\\\nonumber
&\Delta_i \ge
-[d_{i}(k)-d^{'}_{i}(k)-d_{i}(l)+d^{'}_{i}(l)], & 1 \le k < l \le e,i\in [n]\\\nonumber
&d^{'}_{i}(k) - (\lambda_p^i+D^{(p)}(k)) \ge 0, & p=1,2,3,1\le k \le e, i\in [n]\\\nonumber
&d^{'}_{i}(k)-(\lambda_p^i+D^{(p)}(k) )\le u_{pik} \times y_{pik}, & p=1,2,3, 1\le k \le e, i\in [n] \\\nonumber
&\sum_{p=1}^3 y_{pik} \le 2, & 1\le k \le e,  i\in [n] \\\nonumber
&0\le y_{pik}\le 1, y_{pik} \text{ is an integer}, & p=1,2,3, 1\le k \le e, i\in [n] \\\nonumber
&d_i (k) - (\lambda_p^i +  D^{(p)}(k))\ge 0, & p=1,2,3,1\le k \le e, i\in [n] \\\nonumber
&d_{i}(k)-(\lambda_p^i+D^{(p)}(k)) \le v_{pik}\times z_{pik}, & p=1,2,3, 1\le k \le e, i\in [n] \\\nonumber
&\sum_{k=1}^e z_{pik} \le e-1, & p=1,2,3,i\in [n] \\\nonumber
&0\le z_{pik}\le 1, z_{pik} \text{ is an integer},& p=1,2,3, 1\le k \le e, i\in [n] \\\nonumber
\end{align}
where $u_{pik}$ and $v_{kip}$ are large enough constants.
\end{proposition}
\begin{proof}
Our optimization problem can be written more explicitly as
$$\min_{D^{(1)},D^{(2)},D^{(3)}\in \RR^e/\RR{\bf 1}}\sum_{i=1}^{n}\max \{|d_{i}(k)-d^{'}_{i}(k)-d_{i}(l)+d^{'}_{i}(l)|:1 \le k < l \le e\}$$
$$\text{where } d_i^{'} = \lambda_1^i \odot D^{(1)} \oplus \lambda_2^i \odot D^{(2)} \oplus \lambda_3^i \odot D^{(3)}, \text{ with }\lambda_k^i = \min(d_i - D^{(k)}) \text{ and } k=1,2,3.$$
\begin{enumerate}[label=(\roman*)]
\item Define the quantity 
$$\Delta_i = \max \{|d_{i}(k)-d^{'}_{i}(k)-d_{i}(l)+d^{'}_{i}(l)|:1 \le k < l \le e\}, i\in [n].$$
Then the objective function is equivalent to
\begin{align*}
\text{minimize: } & {}\sum_{i=1}^{n}\Delta_i &\\
\text{subject to: } & \Delta_i \ge |d_{i}(k)-d^{'}_{i}(k)-d_{i}(l)+d^{'}_{i}(l)|, & 1 \le k < l \le e.
\end{align*}
These constraints can be reformulated as:
\begin{align*}
\text{subject to: } &\Delta_i \ge d_{i}(k)-d^{'}_{i}(k)-d_{i}(l)+d^{'}_{i}(l), & 1 \le k < l \le e\\
& \Delta_i \ge -[d_{i}(k)-d^{'}_{i}(k)-d_{i}(l)+d^{'}_{i}(l)], & 1 \le k < l \le e.
\end{align*}

\item Recall the definitions
$$d^{'}_i(k)=\max(\lambda_1^i +D^{(1)}(k), \lambda_2^i +D^{(2)}(k), \lambda_3^i +D^{(3)}(k)),$$
\text{where }$\lambda_s^i = \min(d_i - D^{(s)})$.
These are equivalent to
\begin{align*}
d^{'}_i(k)=\text{ maximize: } & \lambda_1^i +D^{(1)}(k), \lambda_2^i+ D^{(2)}(k), \lambda_3^i +D^{(3)}(k), & \\
\text{subject to: } & \lambda_1^i \le d_i (t)- D^{(1)}(t), & t\in [e] \\
& \lambda_2^i \le d_i (t)- D^{(2)}(t), & t\in [e] \\
& \lambda_3^i \le d_i (t)- D^{(3)}(t), & t\in [e].
\end{align*}

\item We can hence divide our original maximization problem into two parts:\\
for all $k=1,2,3,\ldots,e $,
\begin{align*}
d^{'}_i(k)=\text{ maximize: } & \lambda_1^i +D^{(1)}(k), \lambda_2^i+ D^{(2)}(k), \lambda_3^i +D^{(3)}(k), & \\
\text{subject to: } & \lambda_1^i \le d_i (t)- D^{(1)}(t), & t\in [e] \\
& \lambda_2^i \le d_i (t)- D^{(2)}(t), & t\in [e] \\
& \lambda_3^i \le d_i (t)- D^{(3)}(t), & t\in [e].
\end{align*}
and
\begin{align*}
\text{minimize: } & {}\sum_{i=1}^{n}\Delta_i &\\
\text{subject to: } &\Delta_i \ge d_{i}(k)-d^{'}_{i}(k)-d_{i}(l)+d^{'}_{i}(l), & 1 \le k < l \le e\\
& \Delta_i \ge -[d_{i}(k)-d^{'}_{i}(k)-d_{i}(l)+d^{'}_{i}(l)], & 1 \le k < l \le e.
\end{align*}
We can recombine them into one optimization as follows:
\begin{align*}
\text{minimize:} & \sum_{i=1}^{n}\Delta_i & \\\nonumber
\text{subject to: } & \Delta_i \ge
d_{i}(k)-d^{'}_{i}(k)-d_{i}(l)+d^{'}_{i}(l), & 1 \le k < l \le e, i\in [n]\\\nonumber
&\Delta_i \ge
-[d_{i}(k)-d^{'}_{i}(k)-d_{i}(l)+d^{'}_{i}(l)], & 1 \le k < l \le e,i\in [n]\\\nonumber
&d^{'}_{i}(k) \ge \lambda_p^i+D^{(p)}(k), & p=1,2,3, 1\le k \le e, i\in [n]\\\nonumber
&\prod_{p=1}^{3}[d^{'}_{i}(k)-(\lambda_p^i+D^{(p)}(k))]=0, & 1\le k \le e, i\in [n]\\\nonumber
&\lambda_p^i +  D^{(p)}(t)\le d_i (t), & p=1,2,3,t\in [e], i\in [n] \\\nonumber
&\prod_{t=1}^{e}[d_{i}(t)-(\lambda_p^i+D^{(p)}(t))]=0 & p=1,2,3, i\in [n].\nonumber
\end{align*}

\item By adding new binary variables $y_{pik}$ and $z_{pik}$, for each $p\in [3], i\in [n],$ and $k\in [e]$, we can apply the
  Big-M method (an extension of the simplex method \cite{bigM}) to obtain a reformulatation of our problem in terms of mixed integer linear programming:
\begin{align}
\text{minimize }& {}\sum_{i=1}^{n} \Delta_i &\\\nonumber
\text{subject to: } &\Delta_i \ge
d_{i}(k)-d^{'}_{i}(k)-d_{i}(l)+d^{'}_{i}(l), & 1 \le k < l \le e, i\in [n]\\\nonumber
&\Delta_i \ge
-[d_{i}(k)-d^{'}_{i}(k)-d_{i}(l)+d^{'}_{i}(l)], & 1 \le k < l \le e,i\in [n]\\\nonumber
&d^{'}_{i}(k) - (\lambda_p^i+D^{(p)}(k)) \ge 0, & p=1,2,3,1\le k \le e, i\in [n]\\\nonumber
&d^{'}_{i}(k)-(\lambda_p^i+D^{(p)}(k) )\le u_{pik} \times y_{pik}, & p=1,2,3, 1\le k \le e, i\in [n] \\\nonumber
&\sum_{p=1}^3 y_{pik} \le 2, & 1\le k \le e,  i\in [n] \\\nonumber
&0\le y_{pik}\le 1, y_{pik} \text{ is an integer}, & p=1,2,3, 1\le k \le e, i\in [n] \\\nonumber
&d_i (k) - (\lambda_p^i +  D^{(p)}(k))\ge 0, & p=1,2,3,1\le k \le e, i\in [n] \\\nonumber
&d_{i}(k)-(\lambda_p^i+D^{(p)}(k)) \le v_{pik}\times z_{pik}, & p=1,2,3, 1\le k \le e, i\in [n] \\\nonumber
&\sum_{k=1}^e z_{pik} \le e-1, & p=1,2,3,i\in [n] \\\nonumber
&0\le z_{pik}\le 1, z_{pik} \text{ is an integer},& p=1,2,3, 1\le k \le e, i\in [n] \\\nonumber
\end{align}
where $u_{pik},v_{tip}$ are constants, some large enough upper bounds for $d^{'}_{i}(k)-(\lambda_p^i+D^{(p)}(k)$ and $d_{i}(k)-(\lambda_p^i+D^{(p)}(k))$ respectively.

\end{enumerate}
For simplification, we do not explicitly show the constraints on the tropical principal components $D^{(1)},D^{(2)},D^{(3)}$ to be distinct. This could be proved by applying the Big-M method twice. 
\end{proof}

\begin{remark}
Projecting onto a tropical polytope is relatively straightforward compared to projecting onto a tropical linear space (compare Formula \ref{eq:tropproj} and Theorems \ref{blue-rule} and \ref{red-rule}). In theory, one could attempt to reformulate the Stiefel tropical linear space optimization problem from Section \ref{trop:lin} as in Proposition \ref{polytope_projection}; however, the increased complexity of the linear space projection map makes this impractical.
\end{remark}

Due to the large number of variables and constraints involved in Proposition \ref{polytope_projection}, we are able to solve only relatively small cases like Example \ref{polytope_example} below in a reasonable amount of time.

\subsection{Heuristic approximation}
As noted above, the number of variables in the mixed integer linear programming problem in Proposition \ref{polytope_projection} increases quickly with the number of leaves and data points. Because solving mixed linear integer
programming is NP-hard \cite{Lenstra},
this problem is difficult to solve in practice. In analogy with Algorithm \ref{tropical-linear-space-algorithm}, therefore, we develop a heuristic method for approximating the optimal solution for the problem in Proposition 
\ref{polytope_projection}.  
\begin{algorithm}[Approximation for the second order PCA as a tropical polytope]\label{al1} 
\begin{algorithmic}
\State Stochastic optimization algorithm to fit $\mathcal P$ to $D^{(i)}$.
\State Fix an ordered set $V = (D^{(1)}, D^{(2)}, D^{(3)})$ and compute $\mathcal P = \text{tconv}(V)$. 
\Repeat:
	\State Sample three datapoints $D^{(j_1)}, D^{(j_2)}, D^{(j_3)}$ randomly from the set of all datapoints.
	\State Let $V' = \{D^{(j_1)}, D^{(j_2)}, D^{(j_3)}\}$.
	\State Compute $d(\mathcal P') = d(\text{tconv}(V'))$.
	\State if $d(\mathcal P)>d(\mathcal P'),$ set $V \gets V'$.
\Until convergence.
\end{algorithmic}
\end{algorithm}

As before, convergence can be assessed by considering whether a new choice of $V$ has been found over a fixed number of previous iterations. If computational time is limited, another approach might simply be to prespecify a total number $t$ of samples. And of course, when the computational cost is reasonable one could enumerate through all $\binom n 3$ different choices for the generating points of $\mathcal P$ instead of sampling.

\begin{remark}
Three data points $D^{(j_1)}, D^{(j_2)},$ and $D^{(j_3)}$ define both a Stiefel tropical linear space $L_p$ and a tropical polytope $\mathcal P$. Because Stiefel tropical linear spaces are tropically convex, and each of the generating points is contained in $L_p$, we see that $\mathcal P\subseteq L_p$. In particular, given the same convergence criteria, we should expect Algorithm \ref{tropical-linear-space-algorithm} to provide a somewhat better fit than Algorithm \ref{al1}.
\end{remark}

\begin{remark}
Note that Algorithm \ref{al1} is well-suited for applications to phylogenetics. Because $\mathcal U_m$ is a tropical linear space (Theorem \ref{ultrametrics}) and tropical linear spaces are tropically convex, the solution set $\mathcal P = \text{tconv}(D^{(1)},D^{(2)}, D^{(3)})$ obtained from Algorithm \ref{al1} will be contained in the space of ultrametrics. In particular, projections of ultrametrics are also ultrametrics.
\end{remark}



\section{Simulations}\label{sim}

In this section, we apply the previous results to simulated datasets coming from phylogenetics. 

\subsection{Exact methods}

We begin by identifying the exact best-fit tropical polytope with three vertices closest to a small dataset of equidistant trees using Proposition \ref{polytope_projection}. We implemented this proposition mainly based on a {\tt R} interface to the popular optimization software {\tt IBM ILOG CPLEX}, called {\tt cplexAPI}.

\begin{example}\label{polytope_example}
We randomly generated 6 equidistant trees with 3 leaves and computed their vectorized distance matrices in Figure \ref{figure 1} and Table \ref{table 1}. 

\begin{figure}[!ht]
\begin{floatrow}
\ffigbox{%
  \includegraphics[width=0.9\linewidth]{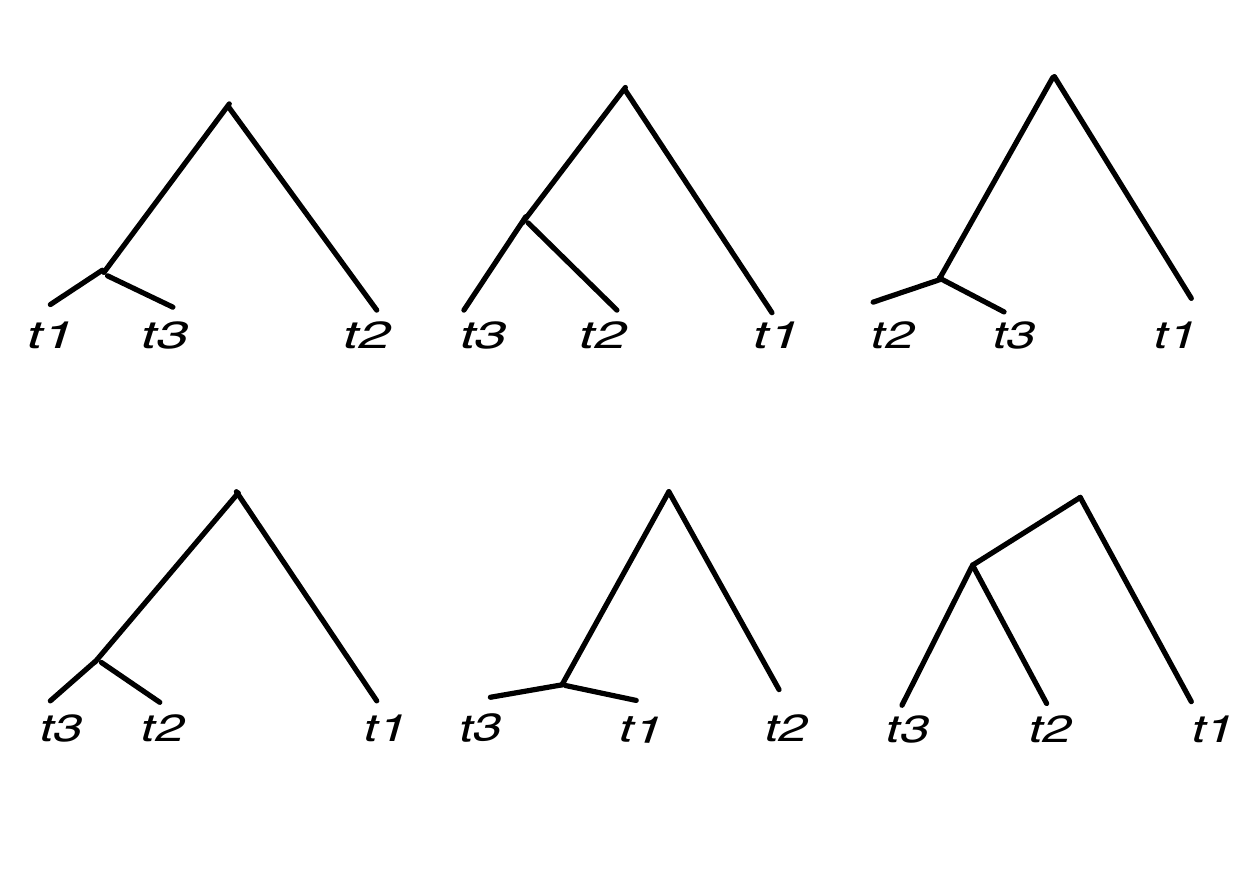}
}{%
  \caption{Random Sample of Trees}%
  \label{figure 1}
}
\capbtabbox{%
  \begin{tabular}[b]{cccc}
	\hline\hline
     tree1 & 0.69089925 & 7.022836 & 7.022836\\
     \hline
     tree2 & 0.53495974 & 1.641369 & 1.641369\\
     \hline
     tree3 & 0.02082164 & 3.101557 & 3.101557\\
     \hline
     tree4 & 0.23519336 & 3.968678 & 3.968678\\
     \hline
     tree5 & 0.19730562 & 5.960980 & 5.960980\\
     \hline
     tree6 & 0.73804678 & 1.090399 & 1.090399\\
     \hline\hline
     \end{tabular}
}{%
  \caption{Vectorized Distance Matrices}%
  \label{table 1}
}
\end{floatrow}
\end{figure}

\begin{figure}[!ht]
\begin{floatrow}
\ffigbox{%
  \includegraphics[width=0.9\linewidth]{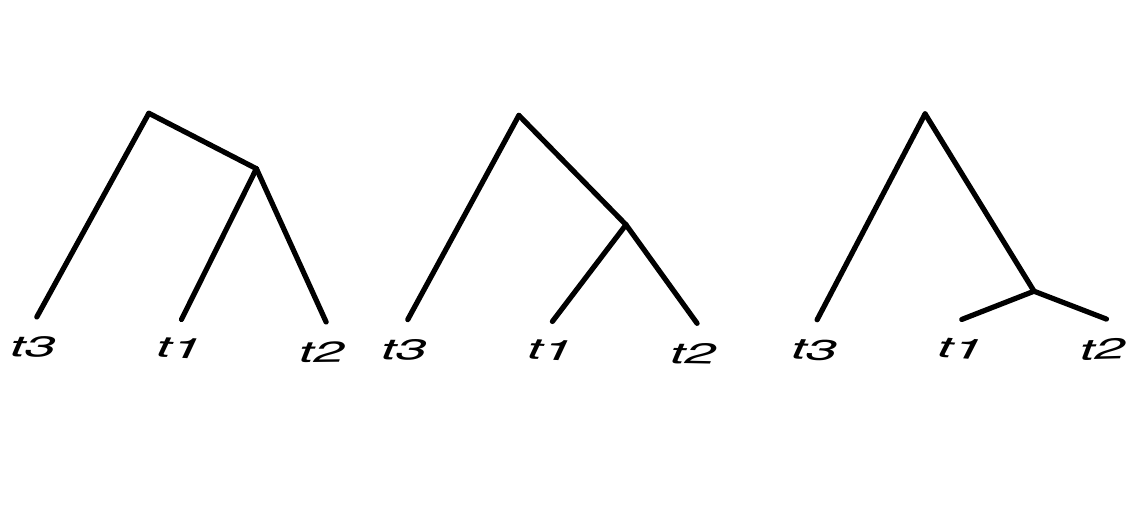}
}{%
  \caption{$D^{(1)},D^{(2)},D^{(3)}$}%
  \label{figure 2}
}
\capbtabbox{%
  \begin{tabular}[b]{cccc}
	\hline\hline
     $D^{(1)}$ & 1 & 1.352352 & 1.352352\\
     \hline
     $D^{(2)}$ & 1 & 2.106409 & 2.106409\\
     \hline
     $D^{(3)}$ & 1 & 7.331937 & 7.331937\\
     \hline\hline
     \end{tabular}
}{%
  \caption{Vectorized Distance Matrices}%
  \label{table 2}
}
\end{floatrow}
\end{figure}

Using our optimization problem formulation from Proposition \ref{polytope_projection}, we obtain $D^{(1)},D^{(2)},D^{(3)}$ for this example. These points are ultrametrics, and they are described in Figure \ref{figure 2} and Table \ref{table 2}. In fact, in this case the best-fit tropical polytope contains all the equidistant trees, so that the sum of distances is zero.
\end{example}

\subsection{Approximative algorithms}
For larger datasets, we turn to the approximative Algorithms \ref{tropical-linear-space-algorithm} and \ref{al1}. We implemented both algorithms in {\tt R}.\footnote{Our software for all computations can be downloaded at \url{http://polytopes.net/computations/tropicalPCA/}.} We then generated a random sample from {\tt Mesquite} \cite{Mesq} and applied our algorithms on this dataset. The sample was constructed as follows:
\begin{algorithm}[Generating the simulation dataset]\label{sim_data} 
                                \qquad  \qquad \qquad 

\begin{enumerate}
\item Generate 250 gene trees with 8 leaves from the coalescent model under a fixed species tree with depth equal to 10 
\item Transform the gene trees to be ultrametrics.
\item Compute approximate second order tropical principal components via the Algorithms \ref{tropical-linear-space-algorithm} and \ref{al1}.
\end{enumerate}
\end{algorithm}

We applied both methods of tropical principal component analysis to a set of random trees generated by Algorithm \ref{sim_data}. In analogy with \cite{NTWY}, we define summary statistics to describe the fit of a Stiefel tropical linear space
or a tropical polytope to a given data
set. If $L_p$ is a Stiefel tropical linear space, we define its distance to the datapoints $d(L_p)$ as
\[d(L_p) = \sum_i d(D^{(i)}, L_p),\]
and a tropical proportion of variance statistic
\[r(L_p) = \frac{\sum_i d_{tr}(\bar \pi, \pi_{L_p}(D^{(i)}))}{\sum_i d_{tr}(D^{(i)}, \pi_{L_p}(D^{(i)})+\sum_i d_{tr}(\bar \pi, \pi_{L_p}(D^{(i)}))}\]
where $\bar \pi$ denotes a Fermat-Weber point of the projections of the datapoints, as in \cite{LY}. These statistics are defined analogously for a tropical polytope $\mathcal P$. The statistic $r(L_p)$ can be interpreted as the proportion of variance explained by $L_p$; in order to remain consistent with the tropical metric, we sum distances rather than squared distances. 

For the polytopal approach, as noted above, the projections will remain ultrametrics. We therefore analyze the topologies of these projections, and compare them with the topology of the species tree.

\subsection{Approximation results}

We applied Algorithm \ref{tropical-linear-space-algorithm} to find an approximate 2-dimensional best-fit Stiefel tropical linear space with a convergence threshold of 100 iterations. The summary statistics for this run were: $d(L_p)=363.0378$ and $r(L_p) = 0.322$.

We also applied a variant of Algorithm \ref{al1} to find an approximate best-fit tropical polytope with three vertices, in which we enumerated through all $\binom{250}{3}$ different choices. The summary statistics were: $d(\mathcal P) = 360.6831$ and 
$r(\mathcal P) = 0.265$. We note that the overall sum of distances is similar between the two methods, but that the best-fit Stiefel tropical linear space explains a slightly higher proportion of variance. 

For the tropical polytope method, we recall that projections of equidistant trees will remain ultrametrics. We present common topologies of the projections as well as the species tree topology
in Figure \ref{sim_topology}.\footnote{Tree topologies of all projected
points can be found in the supplement 
at \url{http://polytopes.net/computations/tropicalPCA/}.}
We observe that these topologies of
projected trees are broadly consistent with the topology of the species tree
under which these gene trees were generated: taxa $g$ and $c$ group
together, as do taxa $h$ and $f$, and the four taxa $a$, $b$, $d$, $e$ also group together. We can view our best-fit tropical polytope as preserving these features of the species tree, meaning that this tropical polytope retains information after projection.
\begin{figure}[!ht]
\centering
\includegraphics[width=0.6\linewidth]{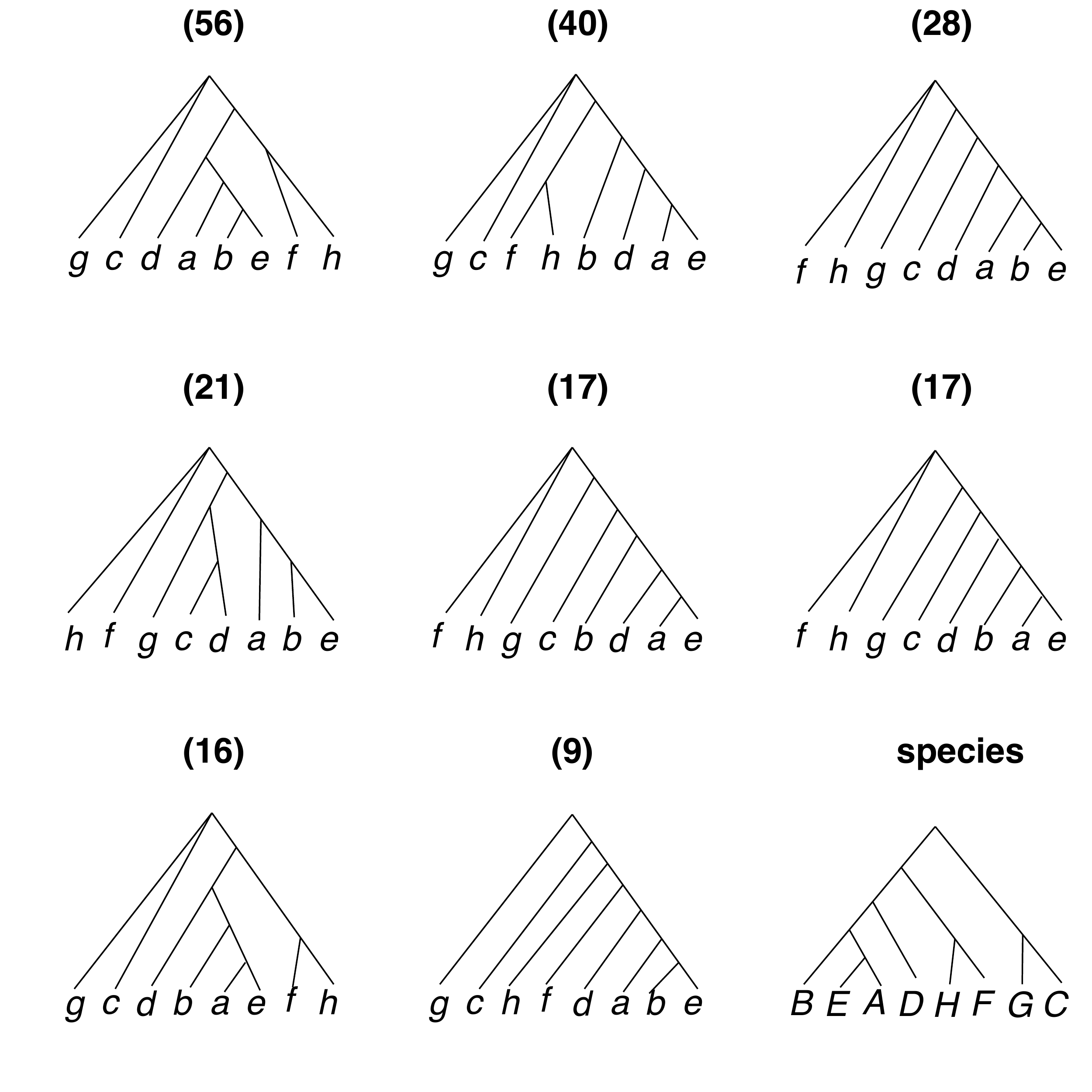}
  \caption{Topology frequencies after projections: the parenthesized numbers are frequencies, and the last tree gives the species tree topology.}
  \label{sim_topology}
\end{figure}

\section{Apicomplexa genome}\label{apicomplexa}

We also applied our tropical principal component algorithms to a set of trees constructed from 252
orthologous sequences on eight species of protozoa in the
Apicomplexa phylum by \cite{KWK}.   
This dataset was also analyzed by Weyenberg et.~al; one can find
more details, such as the gene sequences, in \cite{KDETrees}.
Because ordinary PCA is sensitive to outliers, we removed 16 outlier gene trees identified by \cite{KDETrees} before fitting the tropical principal components.

To find an approximate best-fit 2-dimensional Stiefel tropical linear space, we applied Algorithm \ref{tropical-linear-space-algorithm} with a convergence threshold of 100 iterations. Due to the stochastic nature of the algorithm, we executed the algorithm three times. 
The summary statistics remained consistent between these runs. For one representative execution, these statistics were: $d(L_p) = 145.38$ and $r(L_p)=0.616$. 

We also applied a variant of Algorithm \ref{al1} to find a well-fitted tropical polytope with three vertices, enumerating through all ${\binom{252}{3}}$ possibilities. The summary statistics for this run were: $d(\mathcal P)= 147.0568$ and $r(\mathcal P) = 0.612$. We note that these summary statistics are relatively consistent with the summary statistics obtained from the Stiefel tropical linear space algorithm. 

The tree topologies are presented in Figure \ref{api_topology}. In general, the projected topologies were largely congruent with the generally accepted phylogeny: the two Plasmodium species (Pv and Pf) group together, as do the four species Ta, Bb, Tg, and Et, and Tt is isolated on a deep branch. 

\cite[Theorem 23]{DS} tells us the tropical convex hull of the rows and columns of a matrix are equal. This allows us to visualize our best-fit tropical polytope in the two-dimensional plane $\mathbb R^3/\RR {\bf 1}$ as the tropical convex hull of 28 points. These 28 points divide the polytope into different cells, as described in \cite[Example 9]{JSY}. We plot this polytope, along with its cells and the projections of our data points, in Figure \ref{api_triangle}.  We note that the different topologies seem to divide the tropical polytope PCA into several regions of positive area. 



\begin{figure}[!ht]
\includegraphics[width=1\linewidth]{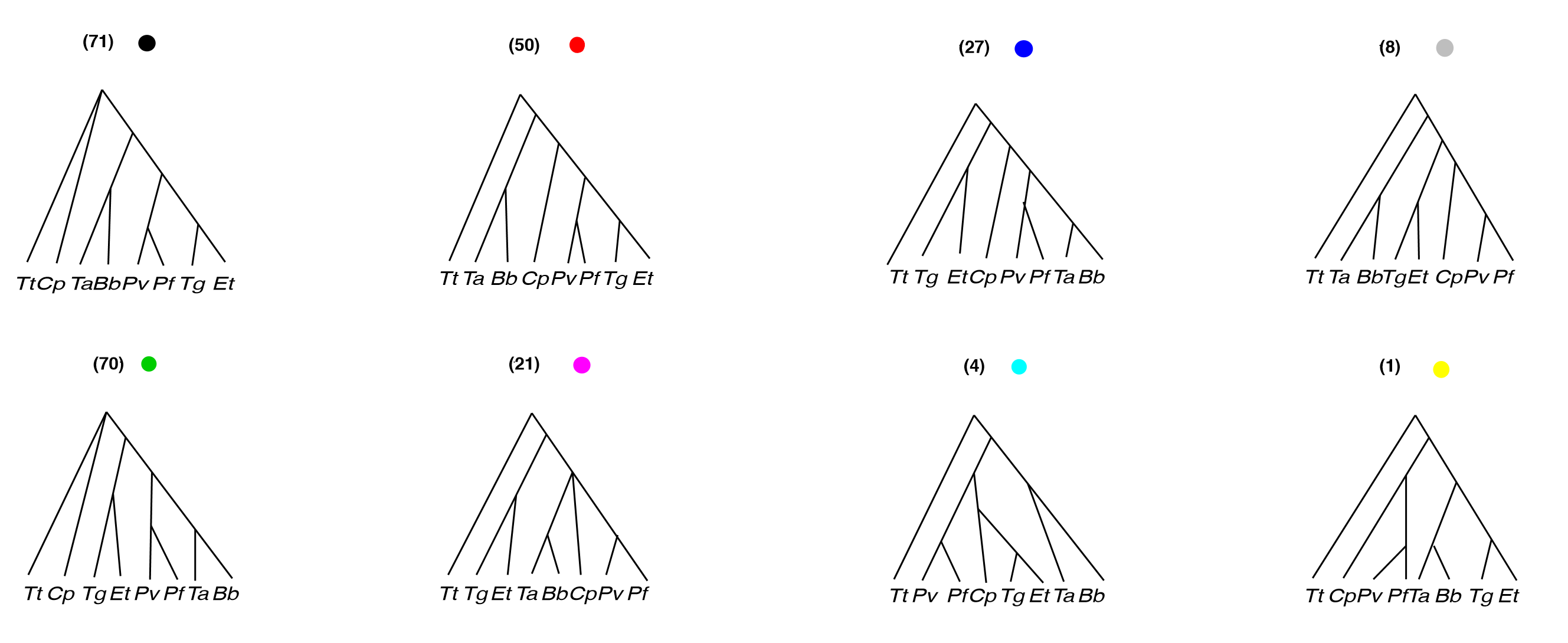}
\caption{Projected topology frequencies from the Apicomplexa dataset: parenthesized numbers give the frequencies of each topology.}%
\label{api_topology}
\end{figure}
\begin{figure}[!ht]

  \includegraphics[width=0.6\linewidth]{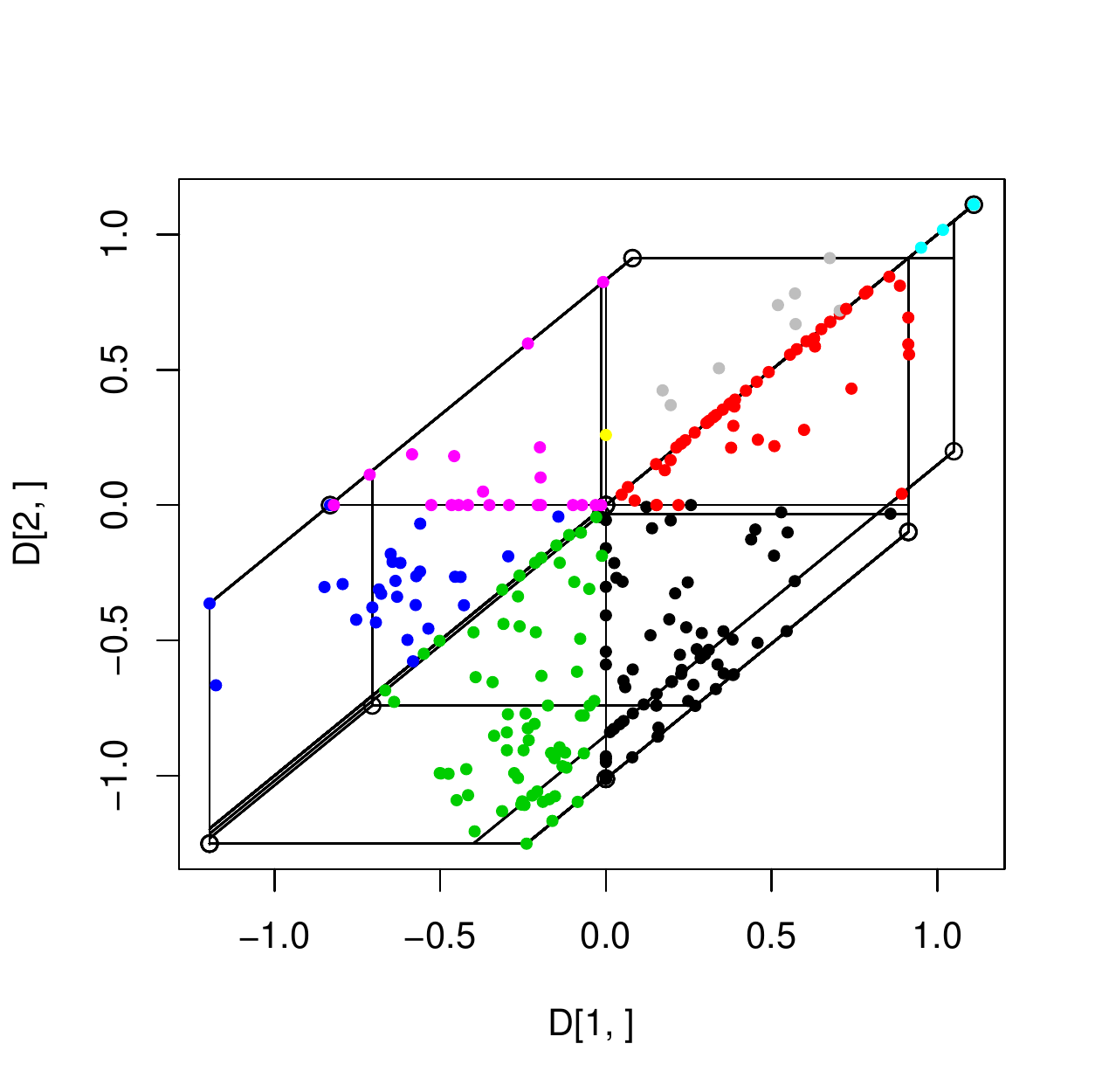}
{%
\caption{Projected points in the tropical polytope PCA, colored as in Figure \ref{api_topology}.}%
\label{api_triangle}
}
\end{figure}

\section{Discussion}\label{discuss}

In recent decades, the field of phylogenetics has found
applications in the analysis on genomic scale data.  In particular, phylogenetic methods have been used to analyze the relationship between species and populations, as well as the evolutionary processes of 
speciation and molecular evolution.  As the cost of generating genomic data continues to decrease, the sheer volume of genomic data demands new analysis techniques. Motivated by this problem from systematic biology, we introduced in this paper a tropical analogue to principal component analysis in the tropical projective torus.

Compared to the classical case, there is still much to be understood about these tropical principal component analyses.  
For example, there is a nested structure to the classical principal components: the zeroth order PCA is contained by the
first order PCA; the first order PCA is
contained by the second order PCA; and so on.  It is unclear whether a similar relationship holds in the tropical analogue, either as a Stiefel tropical linear space or as a tropical polytope. We found examples of best-fit Stiefel tropical linear spaces which do not contain tropical Fermat-Weber points, such as Example \ref{fermat-weber-noncontainment}. In each such case, however, there existed another tropical linear space of equally good fit that did contain a tropical Fermat-Weber point. Because these best-fit tropical structures are in general not unique, it is possible that one could define the principal components so that this containment property holds. Future work could explore this question further.

We also introduced some approximative methods to compute the
second order tropical PCA as a Stiefel tropical linear space and tropical polytope. Both algorithms rely on the uniform sampling of three random points from
the dataset.  However, uniform sampling may not be the most efficient approach to
finding a well-fitted solution. One might explore improvements to
these algorithms using different sampling methods, such as the
Metropolis-Hasting algorithm or Gibb sampling. \cite{ZYCH}

 In \cite{NTWY}, the authors
considered the Billera-Holmes-Vogtman (BHV) \cite{BHV} metric on the
tree space and defined the $(s-1)$st order PCA as the {locus of the weighted Fr\'echet 
  mean} of $s$ distinct points in the tree space. 
Nye et.~al did not use a
convex hull of $s$ distinct points in the tree space under the BHV
metric because Lin et.~al showed in \cite{LSTY} that the dimension of
the convex hull under the BHV metric can be
arbitrary high. In contrast, our methods for tropical principal
component analysis are well-behaved with respect to dimension: the
Stiefel tropical linear space given by an $(s\times e)$-dimensional
matrix will be of dimension $s-1$, and the tropical convex hull of $s$
points has dimension at most $s-1$ as well. \cite[Theorem 5.3.23]{MS} In
statistics, we often use different metrics to analyze
empirical data sets.  Our methods provide a new approach to analyzing phylogenetic tree datasets which may be particularly suitable in certain situations. 

In this work, we also found an exact solution for the best-fit tropical hyperplane of $e$ points in $\RR^e/\RR {\bf 1}$. The general problem of constructing a best-fit Stiefel tropical linear space of given dimension remains unsolved. In addition, given $n$ points in a tropical linear space $L_p$, we noted that the Stiefel tropical linear space defined by these $n$ points may not be contained in $L_p$. Understanding the conditions under which containment holds could enable further application of these techniques to phylogenetics.

\smallskip

\noindent {\bf Acknowledgements.}\smallskip \\
R.~Y.~was supported by Research Initiation Proposals from the Naval
Postgraduate School.  L.~Z.~was supported by an NSF Graduate Research Fellowship. X.~Z.~was supported
by travel funding from the Department of Statistics at the University
of Kentucky.

The authors thank Bernd Sturmfels (UC Berkeley and MPI Leipzig) for many helpful conversations. The authors also thank Daniel Howe (University of Kentucky) for his input on apicomplexa tree topologies.



\end{document}